\documentclass[12pt]{amsart}

\usepackage{mathrsfs}
\usepackage{amsmath,amssymb,enumerate,tikz}
\usepackage{amsfonts,amssymb}
\usepackage{pinlabel}
\usepackage{latexsym,amsfonts,amssymb,verbatim,mathrsfs,amsthm}
\usepackage{amsmath,amsthm,amssymb,latexsym,graphics,textcomp,graphicx,subfigure}
\usepackage{eucal,eufrak}
\usepackage{colonequals}
\usepackage{graphicx}
\usepackage{url}

\theoremstyle{plain}
\newtheorem{theorem}{Theorem}[section]
\newtheorem{lemma}[theorem]{Lemma}
\newtheorem{corollary}[theorem]{Corollary}

\newtheorem{proposition}[theorem]{Proposition}

\theoremstyle{definition}
\newtheorem{remark}{Remark}

\newtheorem*{theorema}{Theorem A }{\bf}{\it}
\newtheorem*{theoremb}{Theorem B }{\bf}{\it}
\newtheorem*{theoremc}{Theorem C }{\bf}{\it}
\newtheorem*{theoremd}{Theorem D }{\bf}{\it}
\newtheorem*{theoreme}{Theorem E}{\bf}{\it}

\newtheorem{defi}{Definition}{\bf}{\it}

\newcommand{\T}{\mathcal{T}_{g,n}(S)}
\newcommand{\mlss}{\mathcal{MLSS}}
\newcommand{\Sim}{\mathcal{S}(S)}
\newcommand{\tw}{{\mathcal{T}}}

\newcommand{\half}{\frac{1}{2}}

\newcommand{\xp}{\mathcal{X}{-piece} }
\newcommand{\yp}{\textit{Y-piece} }
\newcommand{\vp}{\textit{V-piece} }
\newcommand{\jp}{\textit{Joker's hat} }
\newcommand{\yyp}{\textit{YY-piece} }
\newcommand{\yvp}{\textit{YV-piece} }
\newcommand{\yjp}{\textit{YJ-piece} }
\newcommand{\vvp}{\textit{VV-piece} }
\newcommand{\vjp}{\textit{VJ-piece} }
\newcommand{\jjp}{\textit{JJ-piece} }

\begin{document}

\title[finite marked length spectral rigidity]{On finite marked length spectral rigidity of hyperbolic cone surfaces and the Thurston metric}

\author{Huiping Pan}

\address{ School of Mathematical Science, Fudan University, 200433, Shanghai, P. R. China}
\email{panhp@fudan.edu.cn}

\thanks{This work is partially supported by NSFC, No: 11271378.}

\date{\today}

\maketitle

\begin{abstract}
  We study the geometry of hyperbolic cone surfaces, possibly with cusps or geodesic boundaries.
  We prove that any hyperbolic cone structure on a surface of non-exceptional type
  is determined up to isotopy by the geodesic lengths of a finite specific homotopy classes of non-peripheral simple closed curves.
  As an application, we show that the Thurston asymmetric metric is well-defined on the Teichm\"uller space of hyperbolic cone surfaces with fixed cone angles and boundary lengths. We compare such a Teichm\"uller space with the Teichm\"uller space of  complete hyperbolic surfaces
  with punctures, by showing that the two spaces (endowed with the Thurston metric) are almost isometric.
\vskip 10pt
  \noindent Keywords: {finite marked length spectral rigidity; Teichm\"uller space; hyperbolic cone surfaces; Thurston metric.  }\\
 \noindent AMS Mathematics Subject Classification (2010):    53C24,  51F99, 53B40.
\end{abstract}


\section{Introduction}\label{sec:intro}

Let $S=S_{g,n}$ be an oriented surface of genus $g$  with $n$ boundary components.
We shall consider hyperbolic cone metrics on $S$, that is,
hyperbolic structure on $S$ such that each boundary component is either a cone point (with cone angle strictly less than $\pi$),
 a puncture (i.e., a cone point with zero angle), or a simple closed geodesic  with positive length.
 In this case, a boundary component $\Delta$ of $S$ is associated with a   \textit{generalized length function} (also called \textit{boundary assignment}) $\lambda (\Delta)$,
 defined by

  \begin{equation}
    \lambda(\Delta)=\left \{
                            \begin{array}{ll}
                              -\theta, & \text{ if } \Delta \text{ is a cone point of angle } \theta\in(0,\pi),\\
                              0,& \text{ if } \Delta \text{ is a cusp},\\
                              l,& \text{ if } \Delta \text{ is a geodesic boundary of length } l>0.
                            \end{array}
                    \right.
  \end{equation}

\subsection{Parametrization of Teichm\"uller space}
Let  $\mathcal{T}_{g,n}$ be the Teichm\"uller space of marked hyperbolic cone metrics on $S$.
Given any $\Lambda=(\lambda_1,...,\lambda_n)\in (-\pi,\infty)^n$, we let $\mathcal{T}_{g,n}(\Lambda)$ be the subspace
of $\mathcal{T}_{g,n}$  corresponding to hyperbolic cone surfaces whose boundary components  have fixed generalized lengths
$\Lambda$.

It is well known \cite{FLP}  that  the Teichm\"uller space $\mathcal{T}_{g,n}(0)$
can be parameterized by the geodesic lengths of finitely many specified simple closed curves.
Moreover, Sepp\"al\"a and Sorvali \cite{SS2} proved that only $6g-6+2n$ simple closed curves are required if $n\neq0$.
 (For a closed surface of genus $g$,  $6g-5$ simple closed curves are needed,
 see Schmutz \cite{Sch}.)
 A natural question is  whether these results still hold for our enlarged Teichm\"uller space $\mathcal{T}_{g,n}$.
In \S\ref{sec:mlss}, we prove that:

\begin{theorema}
Let $S_{g,n}$ be a surface with $g\geq1,n\geq1$ or $g=0,n\geq 6$.
The Teichm\"uller space $\mathcal{T}_{g,n}$ can be parameterized by the geodesic lengths of
 finitely many non-peripheral simple closed  curves (see Definition \ref{defi-a}) of $S_{g,n}$.
 Moreover, the minimal number of these parameters is less than $12g-12+32n$.
\end{theorema}

If $(g,n)\in\{(0,0),(0,1),(0,2),(0,3),(0,4),(0,5),(1,0)\}$, we call the surface $S_{g,n}$ \emph{exceptional}. Otherwise, we call surface $S_{g,n}$ \emph{non-exceptional}.

\bigskip



 \subsection{Comparisons of geometry between hyperbolic cone surfaces}
Intuitively  a cusp can be considered as a limit of cone points with angles tending to zero,  or as a limit of geodesic boundaries with lengths tending to zero.
We want to compare the geometries of hyperbolic cone surfaces under different assignments of  cone angles and boundary lengths.
In particular, we want to compare the geodesic lengths of simple closed curves on hyperbolic cone surfaces
 when we modify cone points or geodesic boundaries to cusps.

For a given pants decomposition $\Gamma=\{\gamma_1,\cdots,\gamma_{3g-3+n}\}$  and a collection of seams $\mathcal B=\{\beta_1,\cdots,\beta_k\}$, we define $F_{\Gamma,\mathcal B}$ by
    \begin{eqnarray*}
    F_{(\Gamma,\mathcal B)}:\mathcal{T}_{g,n}&\to & \mathcal{T}_{g,n}(0)\\
    (\Lambda,L,T)&\mapsto & (0,L,T),
  \end{eqnarray*}
where $(\Lambda,L,T)$ and $(0,L,T)$ are the corresponding Fenchel-Nielsen (length-twist) coordinates (see Section \ref{sec:fn}).
 We denote the restriction of $F_{\Gamma,\mathcal B}$ on $\mathcal{T}_{g,n}(\Lambda)$
 by $F_{\Gamma,\mathcal B, \Lambda}$.

\begin{theoremb}[Length comparison inequalities]
Let  $S$ be  a non-exceptional surface, $\Gamma=\{\gamma_1,\cdots,\gamma_{3g-3+n}\}$ be a pants decomposition and $\mathcal B=\{\beta_1,\cdots,\beta_k\}$ be a collection of seams. There exist constants $C,D$  depending on $\Lambda$ such that for any $X\in \mathcal{T}_{g,n}(\Lambda)$, $X'=F_{\Gamma,\mathcal B,\Lambda}(X)$ and any isotopy class of non-peripheral simple closed curve $[\alpha]$,
  \begin{equation*}
   \begin{cases}
   |{l_X([\alpha])}-{l_{X'}([\alpha])}|\leq D\sum_{j=1}^{3g-3+n} i([\alpha],[\gamma_j]), \\
  \frac{1}{C}\leq \frac{l_X([\alpha])}{l_{X'}([\alpha])}\leq C,
  \end{cases}
  \end{equation*}
where  $i(\cdot,\cdot)$ is the geometric intersection number, and ${l_X([\alpha])}$ is the length of the geodesic representative in $[\alpha]$.  Moreover, $D\to 0,C\to 1$ as $\Lambda\to 0$.
 \end{theoremb}

\subsection{The Thurston metric on $\mathcal{T}_{g,n}(\Lambda)$}
There is an asymmetric metric on  $\mathcal{T}_{g,n}(0)$ defined by

\begin{equation}\label{eq:metric}
d_{\mathrm{Th}}(X_1,X_2)=\log \sup_{[\alpha]\in\Sim}\frac{l_{X_2}([\alpha])}{l_{X_1}([\alpha])},
\end{equation}
which is the so called \emph{Thurston metric} \cite{Th}.

We combine Theorem A with a generalized Mcshane's identity on a hyperbolic cone surface due to  Tam-Wong-Zhang  \cite{TWZ}
to show that
\begin{theoremc}
  Let  $S$ be a non-exceptional surface. For any $X_1,\ X_2\in \mathcal{T}_{g,n}(\Lambda)$, if $l_{X_1}([\alpha])\geq l_{X_2}([\alpha])\ $ for any $ [\alpha] \in\Sim$, then $X_1=X_2$.
\end{theoremc}
As a result, we can define the Thurston metric on $\mathcal{T}_{g,n}(\Lambda)$ for any $\Lambda\in(-\pi,\infty)^n$, using the formula \eqref{eq:metric}.
It would be interesting to compare the Thurston metric  with the arc metric on $\mathcal{T}_{g,n}(\Lambda)$, defined by Liu-Papadopoulos-Su-Th\'eret
\cite{LPST1}.

Using Theorem B,
we are able to prove that any two Teichm\"uller spaces $\mathcal{T}_{g,n}(\Lambda),\mathcal{T}_{g,n}(\Lambda')$  are almost isometric.
\begin{theoremd}[Almost-isometry]
  Let $\Lambda\in(-\pi,\infty)^n$ and let  $S$ be a non-exceptional surface.
    The map $F_{\Gamma,\mathcal B, \Lambda}:\mathcal{T}_{g,n}(\Lambda)\to \mathcal{T}_{g,n}(0)$ is an almost-isometry, i.e., there is a constant $C$ depending  on $\Lambda$ such that
  $$ d_{\mathrm{Th}}(X'_1,X'_2)-C\leq d_{\mathrm{Th}}(X_1,X_2)\leq d_{\mathrm{Th}}(X'_1,X'_2)+C,\ for\ any\ \  X_1,X_2 \in \mathcal{T}_{g,n},$$
  where $X'_1=F_{\Gamma}(X_1)$ and $X'_2=F_{\Gamma}(X_2)$.  Moreover, $C\to 0$ as $\Lambda\to0$.
  \end{theoremd}

\subsection{The Thurston boundary of $\mathcal{T}_{g,n}(\Lambda)$}
For the space $\mathcal{T}_{g,n}(0)$, the Thurston boundary is naturally identified with $\mathcal{PML}(S)$,
the space of projective classes of measured laminations (\cite{FLP}). We will prove that this is also true in our settings.
Denote by $\mathbb{R}_+^{\mathcal{S}}$ the set of non-negative functionals on the set of isotopy classes of non-peripheral simple closed curves (see Definition \ref{defi-a}).

\begin{theoreme}
   Suppose  $\Lambda\in(-\pi,\infty)^n$ and    $S$ is a non-exceptional surface.  Let $\Psi_{\Lambda}$ and $ \Pi$ be the maps defined as following:
  \begin{eqnarray*}
    \Psi_{\Lambda}:\mathcal{T}_{g,n}(\Lambda)& \longrightarrow &\mathbb{R}^{\mathcal{S}}_+\\
                X &\longmapsto & (l_X([\alpha]))_{\alpha\in\Sim},\\
  \end{eqnarray*}
and
  \begin{eqnarray*}
    \Pi: \mathbb{R}^{\mathcal{S}}_+ &\longrightarrow& P\mathbb{R}^{\mathcal{S}}_+\\
        (s_{\alpha})_{\alpha\in\Sim}&\longmapsto &[(s_{\alpha})_{\alpha\in\Sim}].
  \end{eqnarray*}
  Then
  \begin{enumerate}[(a)]
    \item both $\Psi_{\Lambda}$ and $\Pi\circ\Psi_\Lambda$ are embeddings, where $\mathcal{T}_{g,n}(\Lambda)$ is equipped with the topology induced by $d_{\mathrm{Th}}$ and $\mathbb{R}^{\mathcal{S}}_+$ is equipped with the weak topology;
    \item $\mathcal{T}_{g,n}(\Lambda)\ni X_n \to \xi\in\mathcal{PML}(S) \iff \mathcal{T}_{g,n}(0)\ni F_{\Gamma,\mathcal B,\Lambda}(X_n)\to \xi\in\mathcal{PML}(S)$.
       As a result, the boundary of $\mathcal{T}_{g,n}(\Lambda)$ in $P\mathbb{R}^{\mathcal{S}}_+$ (i.e., the Thurston boundary of $\mathcal{T}_{g,n}(\Lambda)$) is homeomorphic to $\mathcal{PML}(S)$.
   \end{enumerate}
\end{theoreme}

The paper is organized as following. In Section \ref{sec:met}, we recall basic facts about hyperbolic cone surfaces. In Section \ref{sec:fn}, we study the Fenchel-Nielsen coordinates of the Teichm\"uller space $\T$. In Section \ref{sec:rigidity}, we prove Theorem A and Theorem C. In Section \ref{sec:comp-Y}, we compare the lengths of simple closed curves between hyperbolic cone surfaces and prove Theorem B. In Section \ref{sec:al-iso}, we prove Theorem D. In Section \ref{sec:thurstonb}, we study some basic properties of the Teichm\"uller space $\mathcal T_{g,n}(\Lambda)$  and prove Theorem E.

 \textbf{Acknowledgements  }
  I  would like to thank Lixin Liu and Weixu Su for their useful suggestions for improving  this manuscript. I also thank the referee for numerous comments  and suggestions.



 \section{Hyperbolic cone metrics}\label{sec:met}

In this section, we study some basic properties of hyperbolic cone surfaces, with or without boundary, and with each cone angle less than $\pi$.

\begin{defi}[\cite{CHK}]\label{def:cone-surface}
 A \textit{hyperbolic cone-surface} is a two-dimensional manifold $X$, with or without boundary, which can be triangulated by hyperbolic triangles.
\end{defi}
The \textit{singular locus} $Cone(X)$ of a hyperbolic cone-surface $X$ consists of    interior point of $X$ which have no neighbourhoods isometric to a ball in the hyperbolic plane and boundary points which have no neighbourhoods isometric to a half ball in the hyperbolic plane. It follows that
\begin{itemize}
\item $Cone(X)$ is contained in the set of vertices of the hyperbolic triangulation of $X$, and  it is a finite set;
\item At each point of $Cone(X)$, there is a \textit{cone angle} which is the sum of the angles of the dihedral angles of the triangles containing the point;
\item $X\backslash Cone(X)$ has a smooth Riemannian metric of constant curvature $-1$, but this metric is incomplete if there is a cone point with positive cone angle;
\end{itemize}
In this paper, we are interested in the case where $X$ satisfies the following requirements:
\begin{enumerate}[(a)]
  \item X has at most $n$ cone points and  every cone angle is strictly smaller than $\pi$. In particular, a cone point with zero cone angle is called a {\it cusp}.
  \item Every boundary component of $X$ consists of at most one geodesic.
\end{enumerate}

\remark   To distinguish a positive cone angle point and the zero angle point, whenever we mention a cone point we mean a cone point with a positive cone angle.

\begin{defi}\label{defi-a}
  A {\it generalized boundary component} of a  cone surface $X$ is a geodesic boundary component,  a cusp, or a cone point. We denote by $\Sigma_X$ the set of all generalized boundary components of $X$. A  non-trival simple closed curve on $X$ is called  \textit{non-peripheral}  if it is not homotopic to any generalized boundary component. A simple arc with endpoints on the generalized boundary components is called \textit{ non-peripheral} if it is not homotopic to any subarc of  the generalized boundary components.
  \end{defi}

\begin{defi}
 A {\it marked hyperbolic cone surface} is a pair $(f,X)$, where $X$ is a hyperbolic cone-surface, and $f:S\to X$ is a homeomorphism. Two marked hyperbolic cone-surfaces $(f,X)$ and $(f',X')$ are called  equivalent if there is an isometry isotopic to $f\circ (f')^{-1}$.
\end{defi}
A necessary condition of two marked hyperbolic cone surfaces to be  equivalent is that they have the same numbers of cone points, of cusps and of geodesic boundaries.  Denote by $\T$ the space of equivalent classes of marked hyperbolic cone surfaces.

\remark
Since $S$ is an oriented surface obtained by removing $n$ points from a closed surface, a marked hyperbolic cone surface may induce an incomplete metric of curvature $-1$ on $S$. From now on, whenever we talk about a hyperbolic cone metric on $S$, we mean its metric completion.


\vskip 10pt
Next, we collect some basic properties of hyperbolic cone surfaces, for more details, we refer to ~\cite{CHK}, \cite{DP} and \cite{TWZ}.
\begin{proposition}[\cite{DP}\cite{TWZ}]\label{prop:elementary}
Let $X$ be a non-exceptional cone-surface.
\begin{enumerate}[(a)]
\item
Every non-trivial simple closed curve on $X\setminus \Sigma_X$ is freely
homotopic to either a unique simple closed geodesic or a unique cone
point or a cusp. 
\item
If two distinguished, non-trivial closed curves $\alpha$ and $\beta$ intersect $n$ times, their corresponding geodesic will intersect at most n times.
\item
Given two non-intersecting smooth simple curves $\alpha$ and
$\beta$ on $S$ there is at least one geodesic path $c$ between
them such that $d(\alpha,\beta)$ is realized by $c$. Such a path
$c$ is perpendicular to $\alpha$ and $\beta$. If $\alpha$ and
$\beta$ are geodesic, in a free homotopy class of paths with end
points moving on $\alpha$ and $\beta$, such a path $c$ is unique.
This property remains true for singular points in place of one or
both geodesics. \end{enumerate}
\end{proposition}

\begin{theorem}[Pants decomposition~\cite{DP}]\label{thm:collartheorem}
Let $X$ be a non-exceptional cone-surface of genus $g$ with $n$ cone
points $\Delta_1,\ldots, \Delta_n$ .Let $\gamma_1,\ldots,\gamma_m$ be disjoint simple
closed geodesics on $M$. Then the followings hold.
\begin{enumerate}[(a)]
\item $m\leq 3g-3 +n$.\\

\item There exist simple closed geodesics
$\gamma_{m+1},\ldots,\gamma_{3g-3+n}$ which together with
$\gamma_1,\ldots,\gamma_m$ form a partition of $S$.\\







\end{enumerate}

\end{theorem}

Denote by $[\alpha]$ the isotopy class of a simple closed curve $\alpha$.
Denote by $\Sim$ the set of isotopy classes of non-peripheral simple closed curves.
From Proposition \ref{prop:elementary}, it follows that every hyperbolic cone metric on $S$ induces a functional on $\Sim$ defined by
\begin{eqnarray*}
  l:\mathcal{T}_{g,n}&\longrightarrow & \mathbb{R}_+^{\mathcal S}\\
  X&\longmapsto&(l_X([\alpha]))_{[\alpha]\in\Sim},
\end{eqnarray*}
where $l_X([\alpha])$ represents the length of the geodesic representative in $[\alpha]$, and where $ \mathbb{R}_+^{\mathcal S}$ is the space of  nonnegative functionals on $\Sim$. We equip $ \mathbb{R}_+^{\mathcal S}$  with the weak topology.
\begin{defi}\label{def:mls}
  The sequence $(l_X([\alpha]))_{[\alpha]\in\Sim}$ is called {\it the marked length spectrum} about non-peripheral simple closed curves of $X$, denoted by $\mathcal{MLSS}(X)$.
\end{defi}

\subsection{Generalized Y-pieces}\label{sec:yvj}

\begin{figure}

      \subfigure[]
   {
     \begin{minipage}[tbp]{30mm}
     \includegraphics[width=30mm]{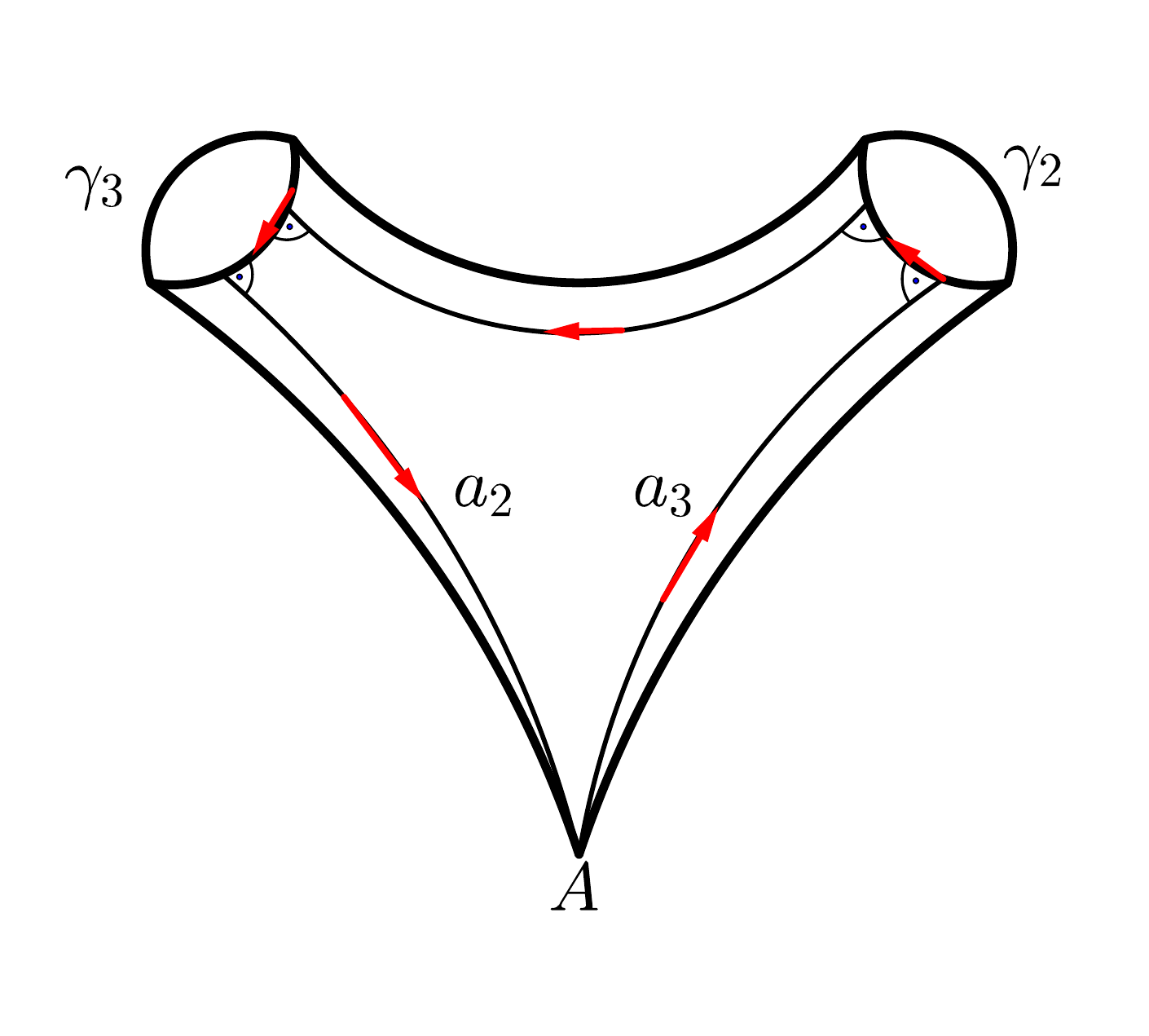}
    \end{minipage}
    \label{fig:V-1}
   }
      \subfigure[]
   {
     \begin{minipage}[tbp]{30mm}
     \includegraphics[width=30mm]{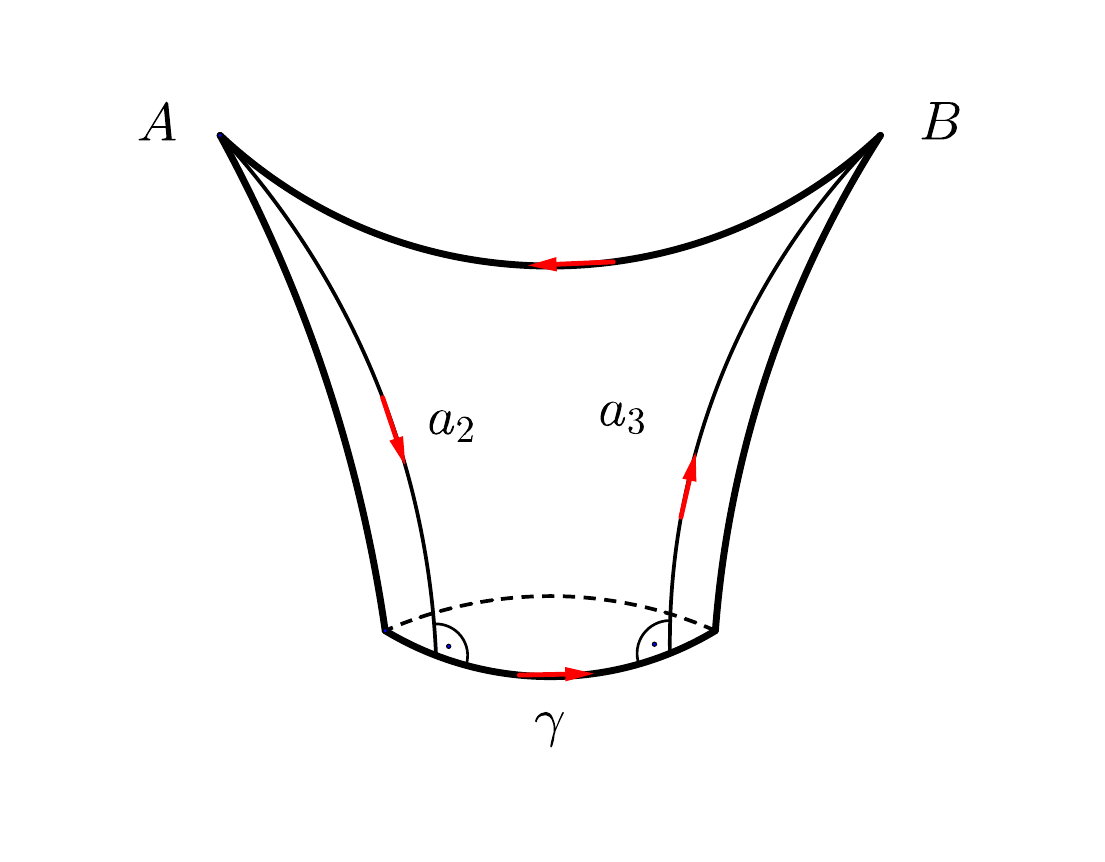}
    \end{minipage}
    \label{fig:J-1}
   }
   \subfigure[]
   {
     \begin{minipage}[tbp]{30mm}
     \includegraphics[width=30mm]{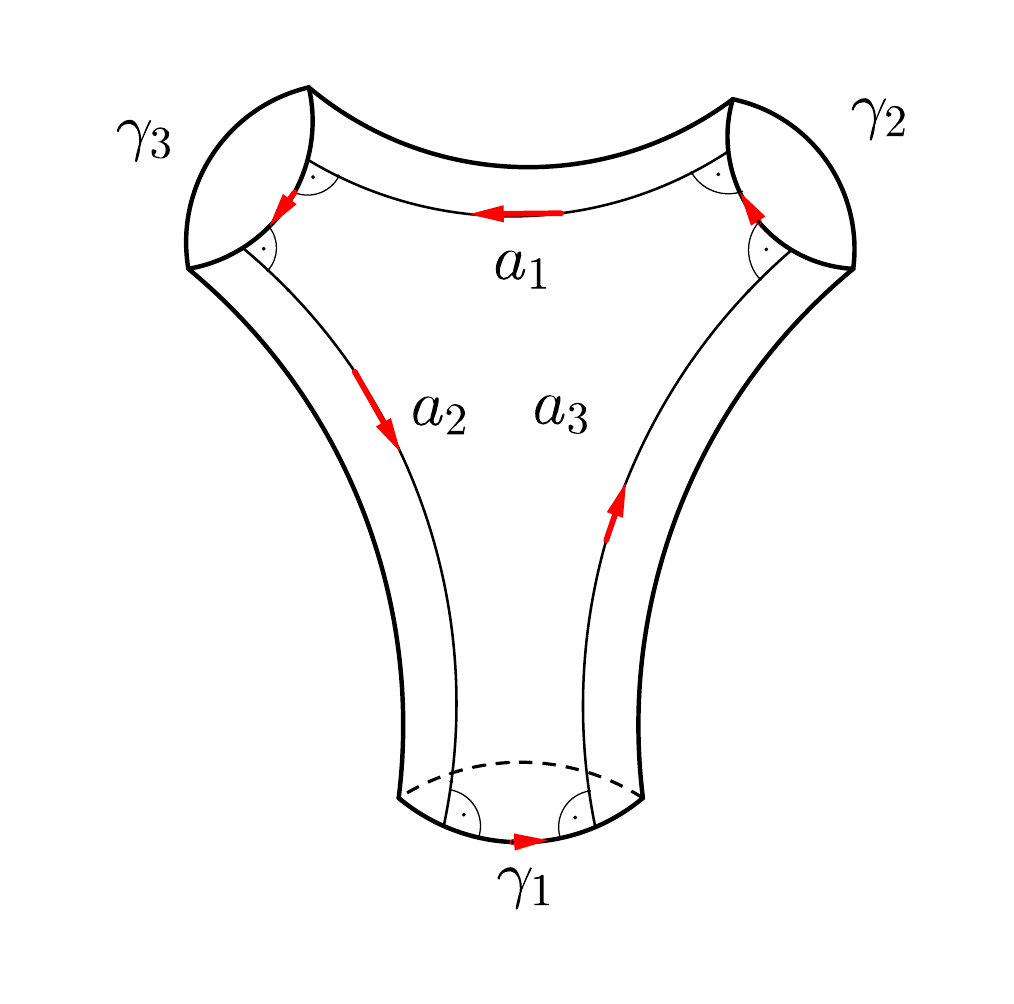}
    \end{minipage}
    \label{fig:Y-piece-1}
   }

  \caption{Generalized Y-pieces. They are V-piece, Joker's hat and Y-piece respectively.}
 \label{fig:Y-piece}
\end{figure}

A \yp  is a sphere with three geodesic boundary components; a \vp   is a sphere with two geodesic boundary components and a cone point with cone angle less than $\pi$, or a sphere with two geodesic boundary components and a cusp; and a \jp  is a sphere with a geodesic boundary component and two cone points with each cone angle less than $\pi$, or a sphere with a geodesic boundary component and two cusps, or a sphere with a geodesic boundary component, a cone point with cone angle less than $\pi$ and a cusp. For convenience,  all these pieces are called {\it generalized Y-pieces}.  or pairs of pants (see Fig.~\ref{fig:Y-piece}).

\subsection{Generalized X-piece}\label{sec:xtw}A generalized $\mathcal X-piece$ is a hyperbolic cone surface obtained by pasting together two generalized Y-pieces along two boundary geodesics of  the same length.
Let $G$ be a V-piece with generalized boundary components $\gamma_1,\gamma_2,\gamma_3$. Let  $G'$ be a Joker's hat with generalized boundary components $\gamma_1',\gamma_2',\gamma_3'$. Assume that $\gamma_1$, $\gamma_1'$ are geodesic boundary components and that they have the same length $l$. Recall that a generalized Y-piece consists of two isometric hyperbolic polygons. Choose an orientation of $\gamma_1$ (resp. $\gamma_1'$) such that $G$ (resp. $G'$) sits on the left. Parameterize $\gamma_1$ (resp. $\gamma_1'$) by arc length such that the basepoint `0' is one of the two vertices  of the corresponding polygon contained in $\gamma_1$ (resp. $\gamma_1$').  We paste $G$ and $G'$ along $\gamma_1,\ \gamma'_1$ with pasting condition:
\begin{equation}\label{idd}
 \gamma_1(s)=\gamma'_1(tl-s),\ s\in\mathbb R/\sim,\ t\in \mathbb R,
\end{equation}
where $\sim$ represents an equivalent relation such that $s\sim s'$ if $s-s'=kl$ for some $k\in\mathbb Z$. 
In this way, we get a \vjp   $ X^t$:
$$ X^t\triangleq G\cup G' \mod(~(\ref{idd})).$$
The curve $\gamma_1$ ($\gamma'_1)$ is called the \textit{waist} of the $VJ\ piece$.

By similar operations, we get \yyp, \yvp, \yjp, \vvp, \jjp. All these pieces are called generalized $\xp$s.

It follows from the pasting condition (\ref{idd}) that $X^{t+1}$ is isometric to $ X^t$, i.e. $t$ is defined only in $\mathbb R/\mathbb Z$.
To extend the domain of $t$ from $\mathbb R/\mathbb Z$ to $\mathbb R$, we need to add some ``marking" to $X^t$ (see \S\ref{sec:fn}).


\subsection{Hyperbolic geometry}
For convenience, we collect some identities of hyperbolic geometry which can be found in \cite{Bu} and \cite{Fen}.
\begin{figure}
    \begin{center}
        \subfigure[]
    {
    \begin{minipage}[tbp]{40mm}
      \includegraphics[width=40mm]{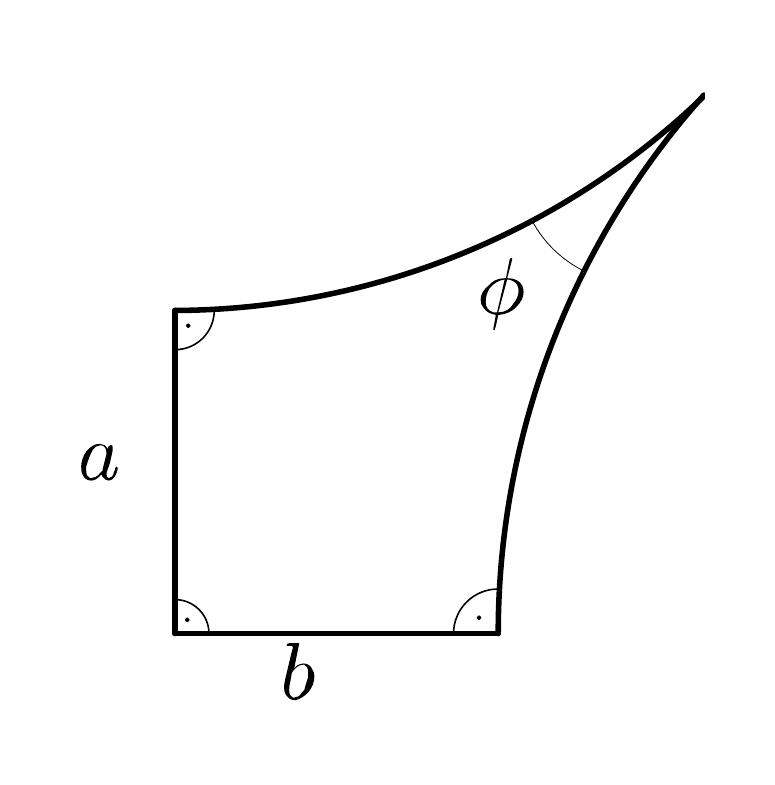}
    \label{fig:tri-rectangle}
    \end{minipage}
    }
            \subfigure[]
    {
    \begin{minipage}[tbp]{40mm}
      \includegraphics[width=40mm]{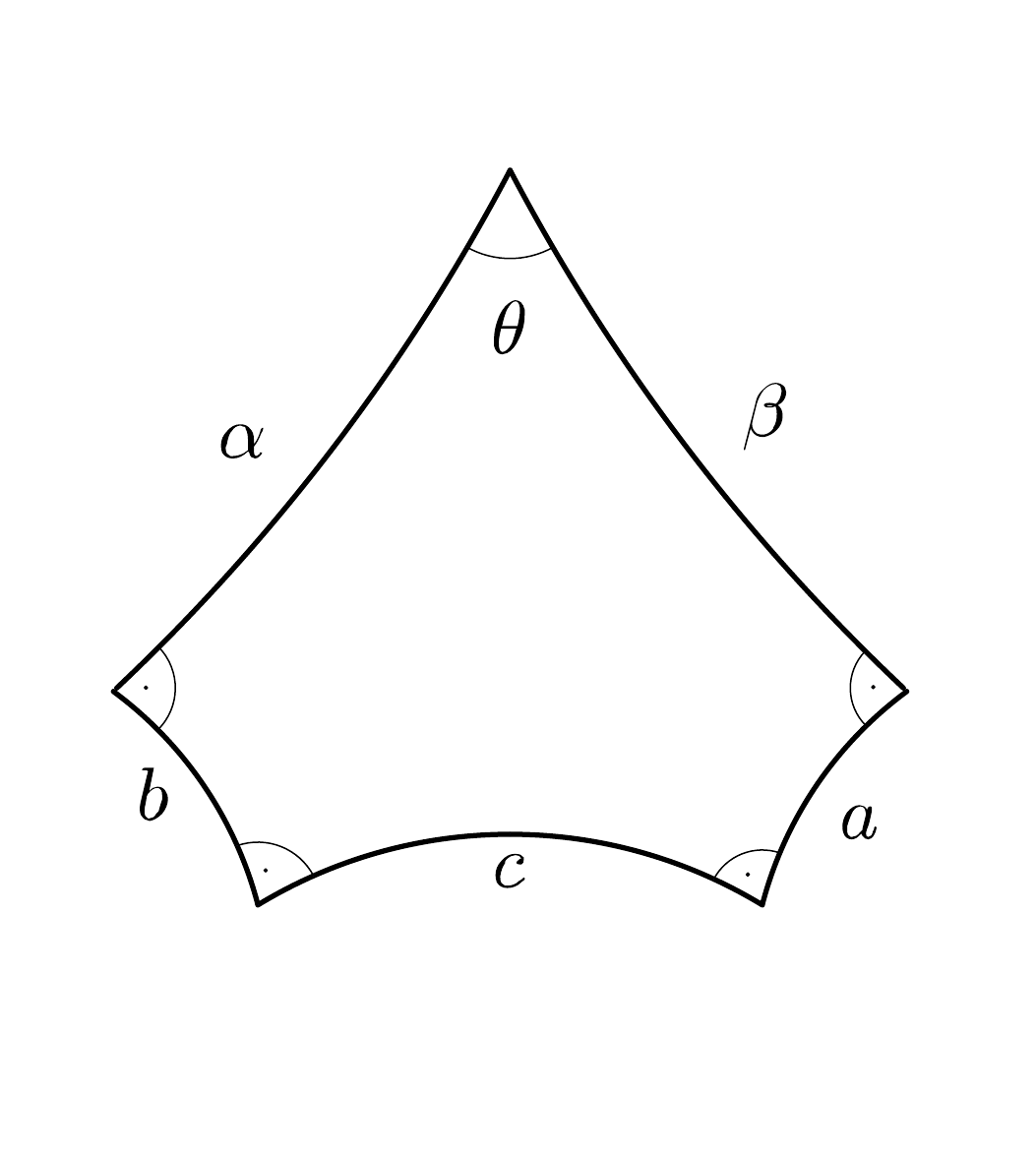}
    \label{fig:Pentagon}
    \end{minipage}
    }
            \subfigure[]
    {
    \begin{minipage}[tbp]{40mm}
      \includegraphics[width=40mm]{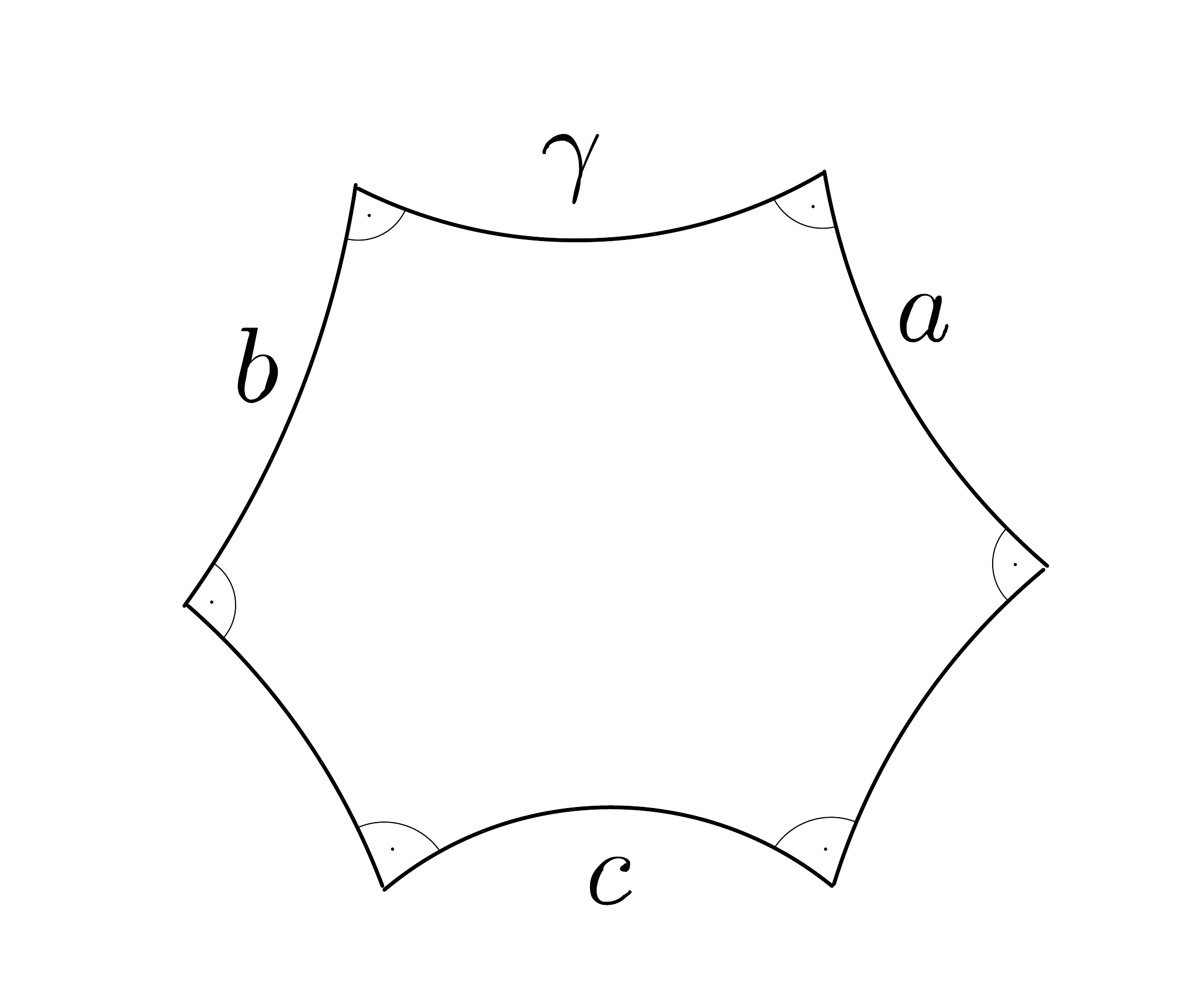}
    \label{fig:Hexagon}
    \end{minipage}
    }
            \subfigure[]
    {
    \begin{minipage}[tbp]{40mm}
      \includegraphics[width=40mm]{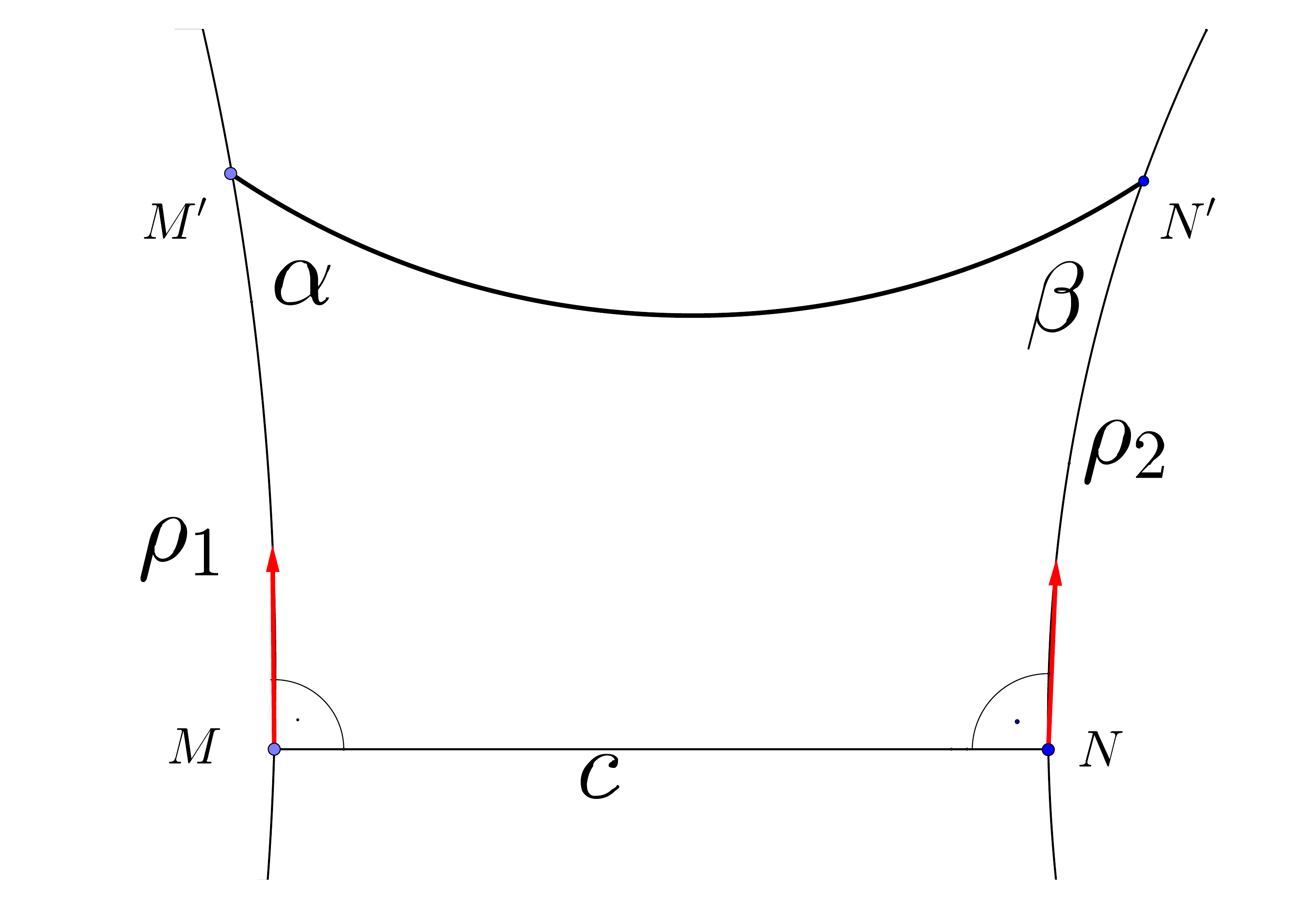}
      \label{fig:Quadri-dist}
    \end{minipage}
    }
                \subfigure[]
    {
    \begin{minipage}[tbp]{40mm}
      \includegraphics[width=40mm]{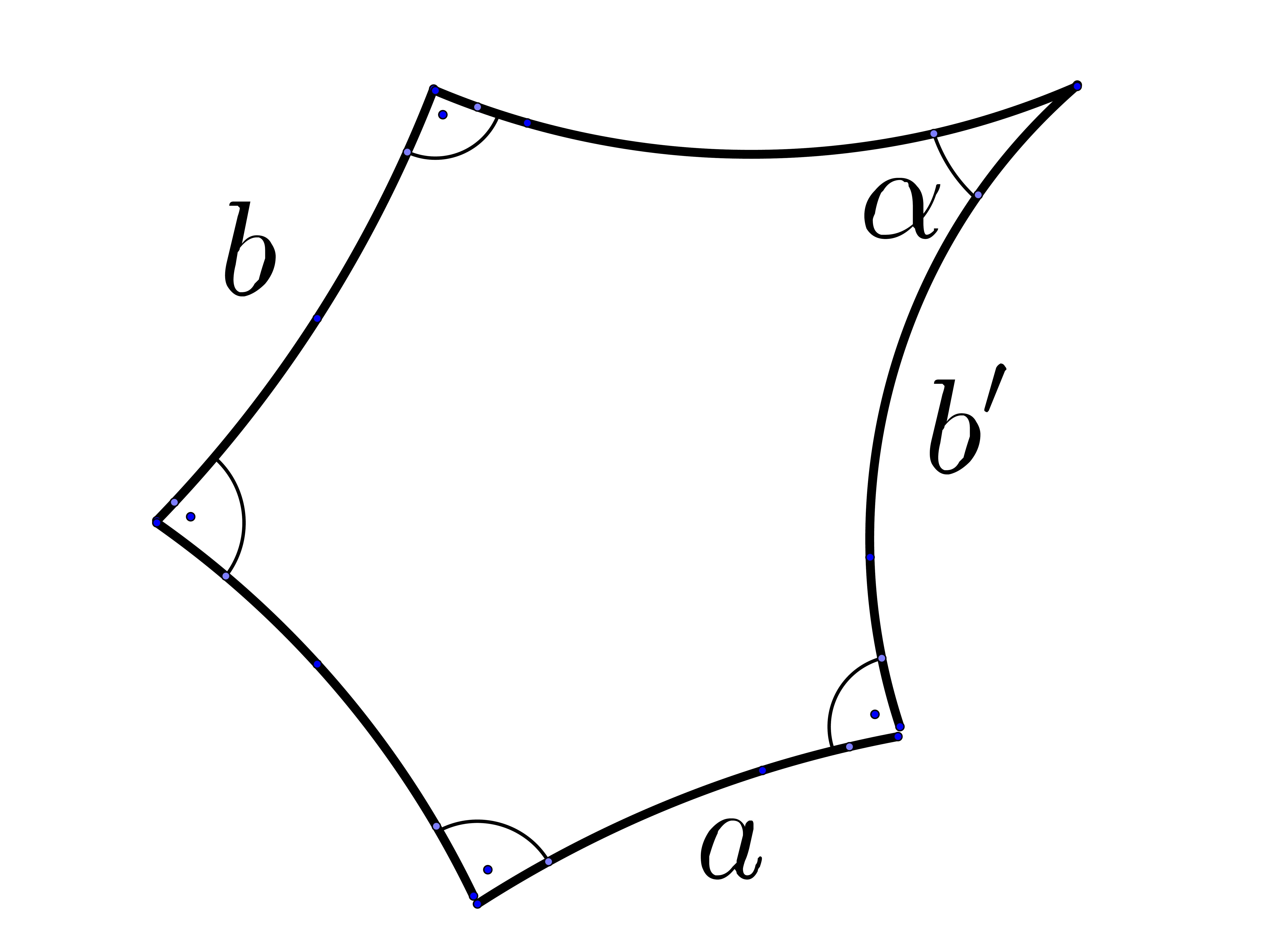}
    \label{fig:pentagon3}
    \end{minipage}
    }
            \subfigure[]
    {
    \begin{minipage}[tbp]{40mm}
      \includegraphics[width=40mm]{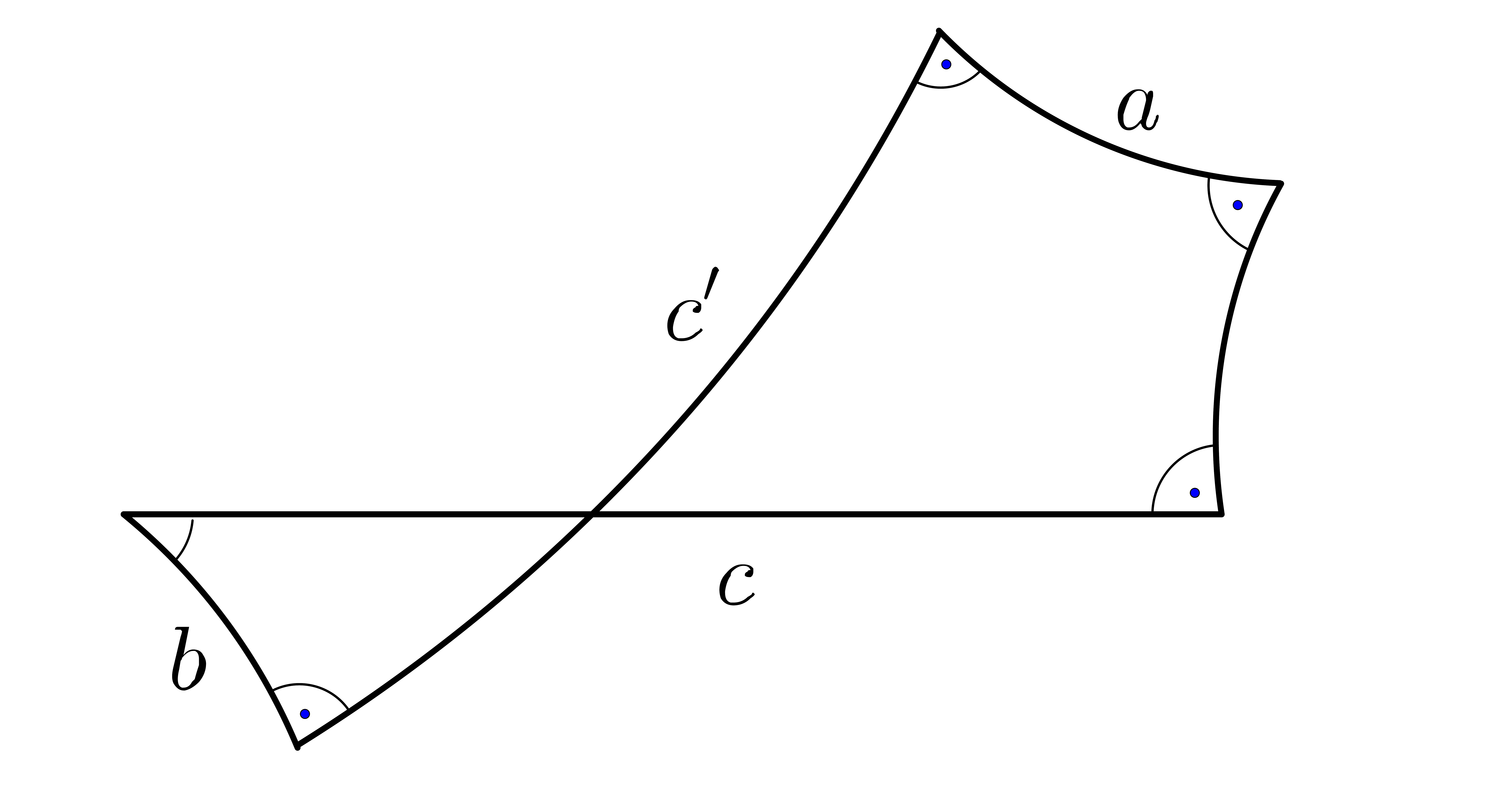}
      \label{fig:self-pentagon}
    \end{minipage}
    }
    \label{fig:ele-formulae}
    \caption{Hyperbolic polygons}
    \end{center}
  \end{figure}
  \begin{lemma}[\cite{Bu},\cite{Fen}]\label{lem:ele-formular}
   Elementary formulae of hyperbolic geometry:
  \begin{itemize}
    \item Tri-rectangle, see Fig.~\ref{fig:tri-rectangle},
          \begin{eqnarray}
              \cos \phi&=&\sinh a \sinh b\label{eq:tri-rec-0}.
          \end{eqnarray}
    \item Pentagon, see Fig.~\ref{fig:Pentagon},
          \begin{eqnarray}
            \cosh c& =& -\cosh \alpha \cosh \beta \cos \theta+\sinh \alpha\sinh\beta,\label{eq:Pen-1}\\
            \cos \theta&=&\sinh a \sinh b \cosh c-\cosh a \cosh b\label{eq:Pen-3}.
          \end{eqnarray}
    \item Hexagon, see Fig.~\ref{fig:Hexagon},
          \begin{eqnarray}
            \cosh c&=& \sinh a \sinh b \cosh \gamma-\cosh a \cosh b.\label{eq:Hex}
          \end{eqnarray}
    \item Quadrilaterals with two right angles, see ~\ref{fig:Quadri-dist},
        \begin{eqnarray}\label{eq:dist}
          \cosh d(M',N')= \cosh{\rho_1}\cosh{\rho_2}\cosh c -\sinh \rho_1 \sinh \rho_2,
        \end{eqnarray}
        where $\rho_1$ (resp.$\rho_2)$ represents oriented distance from $M$ to $M'$ (resp. from $N$ to $N'$). If $\rho_1\rho_2>0$, then
        \begin{equation}\label{eq:dist2}
          \cos \alpha =-\cos \beta \cosh c +\sin \beta \sinh c\sinh |\rho_2|.
        \end{equation}
        and
        \begin{equation}\label{eq:dist3}
          \cosh c =-\cos \alpha\cos \beta +\sin\alpha\sin \beta \cosh d(M',N').
        \end{equation}
    \item  Convex pentagon with four right angles, see Fig.~\ref{fig:pentagon3}
        \begin{eqnarray}
            \cosh b&=& \sin\alpha \sinh a \sinh b'-\cos\alpha \cosh a.\label{eq:pentagon3}
        \end{eqnarray}
    \item The self-intersecting pentagon with four right angles, see Fig.~\ref{fig:self-pentagon}
            \begin{eqnarray}
            \sinh c&=& \sinh\alpha \cosh b \cosh c'+\cosh a \sinh b.\label{eq:self-pentagon}
        \end{eqnarray}
  \end{itemize}

 \end{lemma}



\section{The Fenchel-Nielsen coordinates}\label{sec:fn}
In this section, we  study the Fenchel-Nielsen coordinates for the Teichm\"uller space $\mathcal{T}_{g,n}$. First of all, we need to choose a \textit{coordinate system of curves} on $S_{g,n}$ which consists of the following date:
 \begin{itemize}
   \item a \textit{pants decomposition} which is  a set of non-peripheral, oriented simple closed curves $\Gamma=\{\gamma_1,...,\gamma_{3g-3+n}\}$ such that $[\gamma_i]\neq[\gamma_j]$ if $i\neq j$ and  $S\backslash \Gamma$ consists of pairs of pants $\mathcal{R}=\{R_i\},\ i=1,2,..., 2g-2+n$ (see Fig.~\ref{fig:pants-decomp});
   \item a set of \textit{seams} $\mathcal B=\{\beta_1,\cdots,\beta_k\}$ which is  a collection of disjoint non-peripheral simple closed curves or non-peripheral simple arcs such that the intersection of the union $\cup_{j=1}^{k} \beta_j$ with any pair of pants $R\in\{R_1,\cdots,R_{2g-2+n}\}$ determined by the pants decomposition $\Gamma$ is a union of three disjoint arcs connecting the boundary components of $R$ pairwise.
 \end{itemize}
 Let $R$ be a pair of pants with three oriented boundary  components $\gamma_1,\gamma_2,\gamma_3$, and $\delta$ be the geodesic perpendicular to both $\gamma_1$ and $\gamma_2$. Let $\beta$  be a simple arc on $R$  connecting $\gamma_1$ and $\gamma_2$. There exists a  homotopy $H$ between $\delta$ and $\beta$ which keeps the endpoints on the boundary of $R$. The \textit{twisting number} of $\beta$ at $\gamma_1$ is defined to be the signed displacement from $\delta\cap\gamma_1$ to $\beta\cap\gamma_1$ during the homotopy. The twisting number of $\beta$ at $\gamma_2$ is defined similarly.

  Now we define the \textit{twist parameter}  $t_i$ of $X\in\mathcal{T}_{g,n}$ along $\gamma_i\in\Gamma$. Let $\beta_j\in\mathcal B$ be one of the two seams crossing $\gamma_i$. On each side of the geodesic representative of $\gamma_i$ there is a pair of pants, and the geodesic representative of $\beta_j$ gives an geodesic arc on each of these two pairs of pants. Let $t_{il}$ and $t_{ir}$ be the twisting numbers of each of these arcs on the left and right side of the geodesic representative of $\gamma_i$, respectively. The \textit{twist parameter} of $X$ along $\gamma_i$ is defined to be:
  $$ t_i:=\frac{t_{il}-t_{ir}}{l_X(\gamma_i)}, $$
  where $l_X(\gamma_i)$ represents the length of the geodesic representative of $\gamma_i$ on $X$.

  \begin{remark}\label{remark:twist}
    (1) The sign of the twist parameter of $X$ along $\gamma_i$ does not depend on the orientation of $\gamma_i$. Instead, it depends on the orientation of the  surface.

    (2) The twist parameters depend on the choice of  seams $\mathcal B$. Let $\mathcal B'$ be another set of seams and $\{t_i\}$ be the corresponding twist parameters. Then $t_i-t_i'\in \mathbb Z/2$. Moreover, the difference $t_i-t_i'$ is independent of $X$, it depends only on the seams $\mathcal B$ and $\mathcal B'$.
  \end{remark}
\begin{figure}
  \includegraphics[width=10cm]{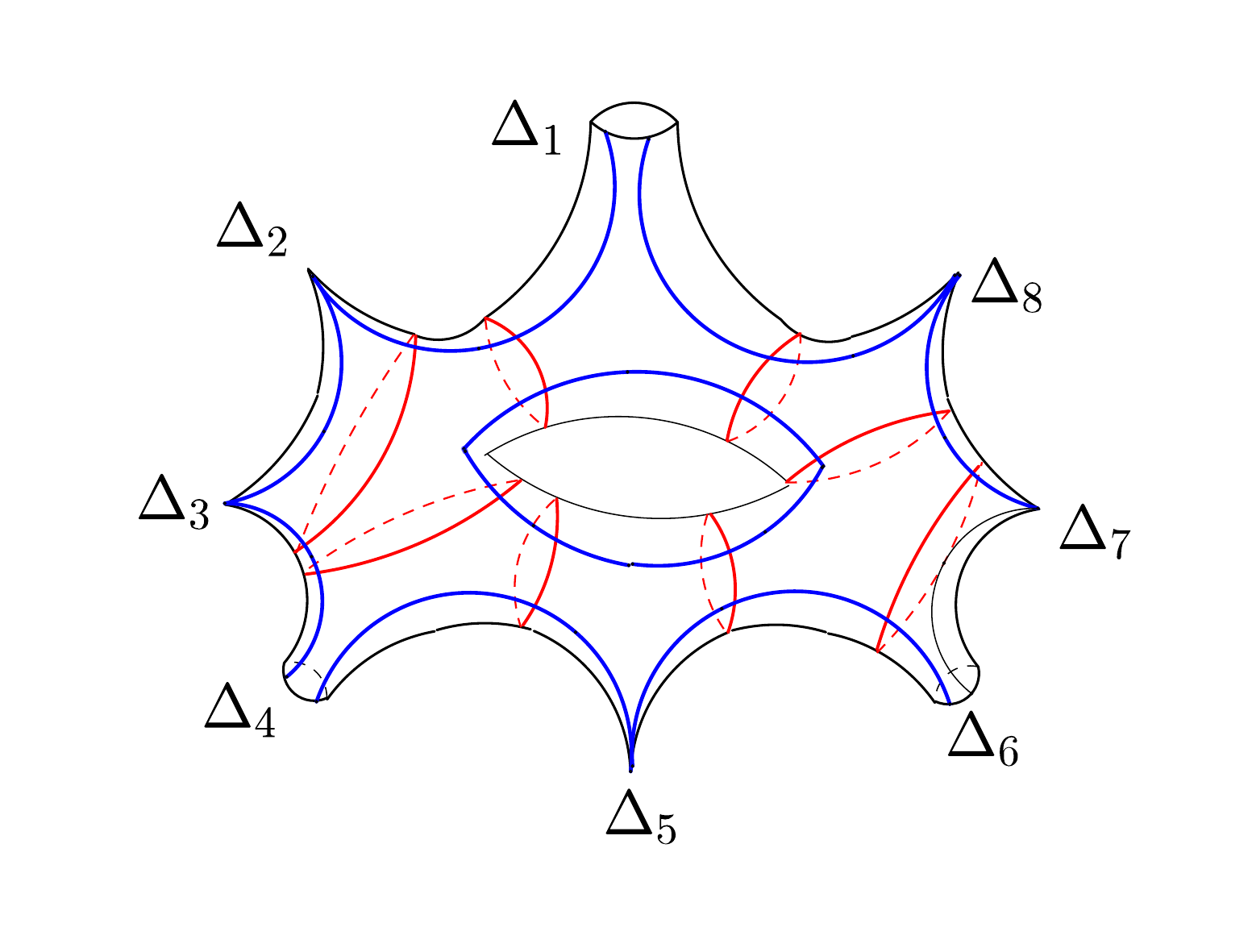}
  \caption{The hyperbolic cone surface $X$ in this figure has one genus and eight generalized boundary components: three geodesic boundaries ($\Delta_1,\Delta_4,\Delta_6$), four cone points ($\Delta_2,\Delta_3,\Delta_7,\Delta_8$) and one cusp ($\Delta_5$). The pants curves (the red circles) decompose $X$ into eight generalized Y-pieces. The set consisting of blue arcs and blue circle consist  a set of seams.}
   \label{fig:pants-decomp}
\end{figure}


Now we fix a \textit{coordinate system of curves} ($\Gamma,\mathcal B$)  and the corresponding pairs of pants $\mathcal{R}$ for $S_{g,n}$. We construct a marked hyperbolic cone surface for any given  $(6g-6+3n)$-tuple $(\Lambda,L,T)\in (-\pi,\infty)^n \times
\mathbb{R}_+^{3g-3+n}\times\mathbb{R}^{3g-3+n}$ as follows:
\begin{itemize}

\item 
$\Lambda=(\lambda_1,\lambda_2,...,\lambda_n)$ is the boundary assignments for the generalized boundary components $\Delta_1,...,\Delta_n$;
\item 
$L=(l([\gamma_1]),l([\gamma_2]),...,l([\gamma_{3g-3+n}])) $ is the  lengths of the pants curves \\
$[\gamma_1],...,[\gamma_{3g-3+n}]\in\Gamma$;
\item 
$T=(t_1,t_2,...,t_{3g-3+n})$ is the twist parameters along $\gamma_1,...,\gamma_{3g-3+n}\in\Gamma$.
\end{itemize}
 $\Lambda$ and $L$ determine the geometry of each pair of pants $R_i\in \mathcal{R}$ and $T$ determines how to paste them. For each $\gamma_i\in\Gamma$, there are two pairs of pants $R,\ R'\in\mathcal{R}$ ($R$ and $R'$ may be the same pair of pants) such that $\gamma_i$ serves as a common boundary component of them. Then $t_i\in T$ determines how $R$ and $R'$ are pasted (for more details about Fenchel-Nielsen coordinates, we refer to \cite[\S10.6]{FM}).

The construction above actually defines a map $\Phi_{\Gamma,\mathcal B}$ as below:
\begin{equation}\label{def:FN}
\begin{array}{rcll}
\Phi_{\Gamma,\mathcal B}:&(-\pi,\infty)^n \times
\mathbb{R}_+^{3g-3+n}\times\mathbb{R}^{3g-3+n}& \longrightarrow &\T\\
&(\Lambda,L,T)&\longmapsto & X.
\end{array}
\end{equation}

In the case that $S$ is a closed surface, $\Phi_{\Gamma}$ is a bijection (also homeomorphism). We will show, by Theorem~\ref{thm:mlss} in \S~\ref{sec:mlss}, that this is also true in our case.
The $6g-6+3n$-tuple $(\Lambda,L,T)$ is called the Fenchel-Nielsen coordinates of $X$ with respect to $(\Gamma,\mathcal B)$, and the map $\Phi_{\Gamma,\mathcal B}$ is called the  Fenchel-Nielsen parametrization of $\T$ with respect to $\Gamma$ and $\mathcal B$.

\begin{figure}

     \subfigure[]
   {
     \begin{minipage}[tbp]{50mm}
     \includegraphics[width=50mm]{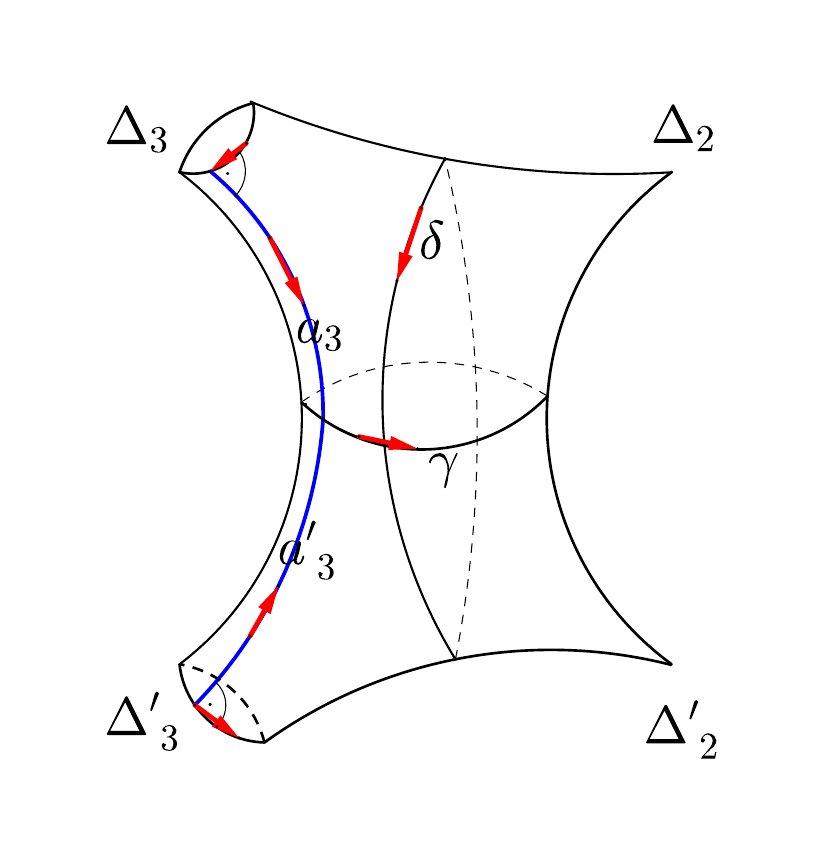}
    \end{minipage}
    \label{fig:V-V0}
   }
   \begin{minipage}[tbp]{15mm}
    \begin{tikzpicture}
      \draw [->](-0.5,0)--(0.5,0)
         node[pos=0.5,above] {twist $t$ along }
         node[pos=0.5,below] {$\gamma$};
    \end{tikzpicture}
    \end{minipage}
      \subfigure[]
   {
     \begin{minipage}[tbp]{50mm}
     \includegraphics[width=50mm]{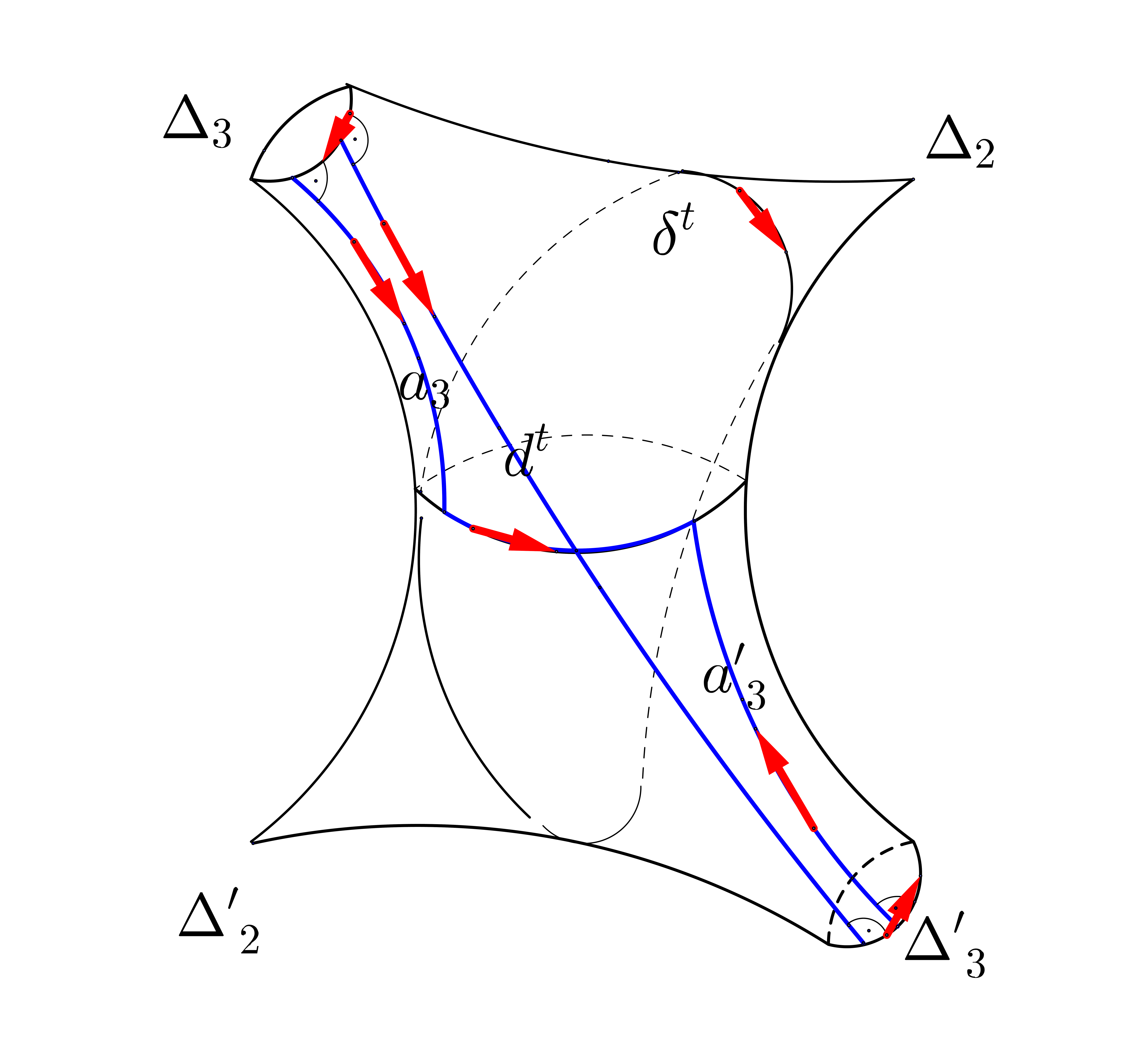}
    \end{minipage}
    \label{fig:V-Vt}
   }

  \caption{Constructions of generalized $\xp s$ by gluing two pairs of pants along two chosen boundaries and twisting to the left about length $t|\gamma|$. Orientations of curves: $X$ sits on the left of $\Delta_3$ and $\Delta_3'$, the cone points $\Delta_2,\Delta_2'$ sit on the left of $\delta$, $\Delta_2,\Delta3$ sit on the left of  $\gamma$, the orientation of $a_3$ (resp.  $a_3'$ ) is chosen from $\Delta_3$ (resp. $\Delta_3'$) to $\gamma$.}
 \label{fig:X-piece-1}
\end{figure}
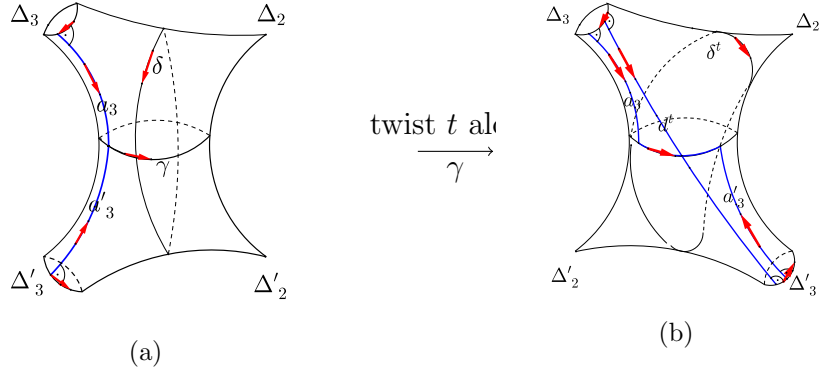

\subsection{Boundary assignments and twist of a generalized X-piece}

First, we consider the boundary assignments of a generalized $\xp$.
\begin{lemma}[Boundary assignments]\label{lem:angsim}
  Let $X$ be a generalized  $\xp$ with generalized boundary components $ \Delta_2, \Delta'_2,\Delta_3, \Delta'_3$. If $ \Delta_3, \Delta'_3$ are geodesic boundaries with given lengths $\lambda_3,\lambda'_3$, then the boundary assignments of $\Delta'_2,\Delta_2$ are determined by $\lambda_3,\lambda'_3$ and $\mlss (X)$.
   Moreover, we need at most 28 fixed non-peripheral simple closed curves.

\end{lemma}

\begin{proof}
A  generalized boundary can be a cone point, a cusp, or a closed geodesic. First, we consider the case that the $\Delta_2,\Delta'_2$ are cone points or cusps (see Fig.~\ref{fig:V-V0}). 

Let $\gamma$ be a non-peripheral simple closed curve of $\mathcal X$ and $\mathcal B=\{\beta_1,\beta_1',\beta_2,\beta_2'\}$ a collection of seams, where $\beta_1$ (resp. $\beta_1'$) connects $\Delta_2$ and $\Delta_3$ (resp. $\Delta_2'$ and $\Delta_3'$) and $\beta_2$ (resp. $\beta_2'$) connects $\Delta_3$ and $\Delta_3'$ (resp. $\Delta_2$ and $\Delta_2'$). Let $\delta$ be the simple closed curve homotopic to $\beta_2\cdot\gamma'_3\cdot(\beta_2)^{-1}\cdot\gamma_3$.  The waist $\gamma$ cuts $X$ into two pairs of pants with the third boundaries $\Delta_1,\Delta'_1$, respectively. Let $X_t$ be the generalized $\xp$ obtained by gluing these two pairs of pants along $\Delta_1,\Delta'_1$ with twist amount $t$ with respect to $\mathcal B$  (see  Fig.~\ref {fig:V-Vt}).

Denote by  $Tw^n_\gamma$  the $n$ times Dehn twist along $\gamma$. Let $\delta_n$ be  the simple closed curve obtained from $\delta$ by $n$ times Dehn twist along $\gamma$, i.e. $\delta_n=Tw^n_\gamma \delta$. Denote by $a_3$ (resp. $a'_3$ ) the geodesic perpendicular to both $\Delta_3$  and $\Delta_1$ (resp. $\Delta'_3$ and $\Delta'_1$).  All the directions (except for $\Delta_1,\Delta'_1$) are illustrated in Fig.~\ref{fig:X-piece-1}. Let $d^t$ and $\delta^t$ be the  geodesic representatives of $\beta_2$ and $\delta$ on $X_t$, respectively. 

With the notations above, we have the following formula which can be found in \cite[Prop.3.3.11]{Bu}
\begin{equation}\label{eq:len-delta}
\begin{array}{rcl}
\cosh{\half|\delta_n|}&=&\sinh(\half\lambda_3)\sinh(\half\lambda_3')\{\sinh {|a_3|}\sinh{|a_3'|}\cosh(t+n)|\gamma|\\
&&+\cosh{|a_3|}\cosh{|a_3'|}\}
-\cosh{(\half\lambda_3)}\cosh{(\half\lambda_3')},
\end{array}
\end{equation}
where $|\delta_n|, |a_3|,|a_3'|,|\gamma|$, $\lambda_3,\lambda_3'$ represent the lengths of the corresponding geodesic representatives.

Hence
\begin{equation}\label{eq:diff}
\cosh\frac{|\delta_{1}|}{2}-\cosh\frac{|\delta|}{2}=
\sinh\frac{\lambda_3}{2}\sinh\frac{\lambda'_3}{2}\sinh {|a_3|}\sinh{|a_3'|}(\cosh(t|\gamma|+|\gamma|)-\cosh(t|\gamma|)),
\end{equation}
\noindent \textup{and}
\begin{equation}\label{eq:time}
\frac{\cosh\frac{|\delta_2|}{2}-\cosh\frac{|\delta_1|}{2}}
{\cosh\frac{|\delta_1|}{2}-\cosh\frac{|\delta|}{2}}=
\frac{\cosh(t|\gamma|+2|\gamma|)-\cosh(t|\gamma|+|\gamma|)}
{\cosh(t|\gamma|+|\gamma|)-\cosh(t|\gamma|)}.
\end{equation}
From  this we know that $t$  is determined by $|\delta|,\ |\delta_1|,|\ \delta_2|$ and $|\gamma|$.

Squaring (\ref{eq:Pen-3}) and rearranging, we have
\begin{eqnarray}\label{eq:aaa}
  \sinh^2\frac{\lambda_3}{2}\sinh^2 {|a_3|}
  &=& \sinh^{-2}\frac{|\gamma|}{2}[\cosh^2 \frac{|\gamma|}{2}+2\cos\frac{\lambda_2}{2}\nonumber
  \cosh\frac{\lambda_3}{2}\cosh\frac{|\gamma|}{2}\\
  &&\qquad +
  \cos^2\frac{\lambda_2}{2}+\cosh^2\frac{\lambda_3}{2}-1],
\end{eqnarray}
where $-\lambda_2$ is the cone angle at the cone point $\Delta_2$.

Set
\begin{eqnarray*}
  B_2\triangleq2\cos\frac{\lambda_2}{2}
  \cosh\frac{\lambda_3}{2},&&  C_2\triangleq\cos^2\frac{\lambda_2}{2}+\cosh^2\frac{\lambda_3}{2}-1;\\
    B'_2\triangleq2\cos\frac{\lambda'_2}{2}
  \cosh\frac{\lambda'_3}{2},&&  C'_2\triangleq\cos^2\frac{\lambda'_2}{2}+\cosh^2\frac{\lambda'_3}{2}-1.
\end{eqnarray*}Then
\begin{eqnarray}\label{eq:BC}
 \nonumber &&\frac{(\cosh\frac{|\delta_{1}|}{2}-\cosh\frac{|\delta|}{2})^2 \sinh^{4}\frac{|\gamma|}{2}}
  {[\cosh(t|\gamma|+|\gamma|)-\cosh(t|\gamma|)]^2}\\
  &=& \cosh^{4}\frac{|\gamma|}{2}+(B_2+B_2')\cosh^3\frac{|\gamma|}{2}
  +(C_2+C_2'+B_2B_2')\cosh^{2}\frac{|\gamma|}{2}\\
\nonumber  &&+(B_2C_2'+B_2'C_2)\cosh\frac{|\gamma|}{2}+C_2C_2'.
\end{eqnarray}
Equation (\ref{eq:BC}) is a linear equation about parameters $B_2+B_2'$, $C_2+C_2'+B_2B_2',$
 $B_2C_2'+B_2'C_2$ and $C_2C_2'$.

 Now, changing the  pants decomposition curve from $\gamma$ to $\gamma_k=Tw^k_\delta \gamma$, $k=\pm1,$ $\pm2,$ $\pm3$,  where  $Tw^n_\delta$  represents the $n$ times Dehn twist along $\delta$,
we get six more  linear equations about parameters $B_2+B_2'$, $C_2+C_2'+B_2B_2'$
, $B_2C_2'+B_2'C_2$ and $C_2C_2'.$

Combining  (\ref{eq:pentagon3}) and (\ref{eq:self-pentagon}), we have the following formula
\begin{equation}\label{eq:len-gamman}
\begin{array}{rcl}
\cosh{\half|\gamma_k|}&=&\sinh(\half\lambda_3)\sin(\half\lambda_2)\{\sinh {|b_3|}\cosh{|b_2|}\cosh(\tilde{t}+k)|\delta|\\
&&+\cosh{|b_3|}\sinh{|b_2|}\}
-\cos{(\half\lambda_2)}\cosh{(\half\lambda_3)},
\end{array}
\end{equation}
where $b_3$ (resp. $b_2$) is the length of the geodesic perpendicular to both $\Delta_3$ and $\delta$ (resp. $\Delta_2$ and $\delta$), $\tilde t$ represents the twist of  $X$ along $\delta$.

It follows from (\ref{eq:len-gamman}) that if $\tilde t<0$, then $|\gamma_{-3}|$, $|\gamma_{-2}|$,$|\gamma_{-1}|$ and $|\gamma_{0}|$ are pairwise different,  if $\tilde t\geq0$, then $|\gamma_{3}|$, $|\gamma_{2}|$,$|\gamma_{1}|$ and $|\gamma_{0}|$ are pairwise different. Hence, at least one of the following two matrices has non-zero determinant,
\begin{equation*}
   \begin{pmatrix}
    \cosh^3\frac{\gamma_{3}}{2}
   & \cosh^2\frac{\gamma_{3}}{2}
   &\cosh\frac{\gamma_{3}}{2}
   &1 \\
   \cosh^3\frac{\gamma_{2}}{2}
   & \cosh^2\frac{\gamma_{2}}{2}
   &\cosh\frac{\gamma_{2}}{2}
   &1 \\
   \cosh^3\frac{\gamma_{1}}{2}
   & \cosh^2\frac{\gamma_{1}}{2}
   &\cosh\frac{\gamma_{1}}{2}
   &1 \\
   \cosh^3\frac{\gamma_{0}}{2}
   & \cosh^2\frac{\gamma_{0}}{2}
   &\cosh\frac{\gamma_{0}}{2}
   &1  \end{pmatrix},\quad
  \begin{pmatrix}
       \cosh^3\frac{\gamma_{-3}}{2}
   & \cosh^2\frac{\gamma_{-3}}{2}
   &\cosh\frac{\gamma_{-3}}{2}
   &1 \\
   \cosh^3\frac{\gamma_{-2}}{2}
   & \cosh^2\frac{\gamma_{-2}}{2}
   &\cosh\frac{\gamma_{-2}}{2}
   &1 \\
   \cosh^3\frac{\gamma_{-1}}{2}
   & \cosh^2\frac{\gamma_{-1}}{2}
   &\cosh\frac{\gamma_{-1}}{2}
   &1 \\
   \cosh^3\frac{\gamma_{0}}{2}
   & \cosh^2\frac{\gamma_{0}}{2}
   &\cosh\frac{\gamma_{0}}{2}
   &1  \end{pmatrix}.
\end{equation*}
It follows that the there is a unique solution for the 4-tuple $(B_2+B_2' C_2+C_2'+B_2B_2'
, B_2C_2'+B_2'C_2, C_2C_2')$.
Further, the boundary assignments $\lambda_2, \lambda'_2$ can be uniquely obtained from $B_2+B_2'$ and
$C_2C_2'$.

If $\delta_2$ (resp. $\delta_2'$) is a geodesic, we need to find the corresponding equations similar to (\ref{eq:time}) and (\ref{eq:aaa}). Repeating the  calculations above using (\ref{eq:pentagon3}) and (\ref{eq:self-pentagon}), we find that the equation corresponding to (\ref{eq:time}) is exactly the same as (\ref{eq:time}), while the equation corresponding to (\ref{eq:aaa}) is a little bit different from (\ref{eq:aaa}) that
 $\cos \frac{\lambda_2}{2}$  is replaced by $\cosh \frac{\lambda_2}{2}$.

It  follows  from  equation (\ref{eq:time}) that, for each one of the seven curves $\{Tw_\delta^i\gamma: i = 0,\pm1, \pm2,\pm3\}$, four curves (including $Tw_\delta^i $ itself) are involved to determine the twist parameter. Therefore, we need at most 28 curves.

\end{proof}

\vskip 10pt
The proof of Lemma \ref{lem:angsim} also proves that the twist amount of a generalized $\xp$ with respect to a waist and a collection of seams is determined by its marked length spectrum.
\begin{lemma}[Twist]\label{lem:twist}
Let $X$ be a generalized $\xp$ with generalized boundary components $\Delta_2,\Delta'_2, \Delta_3,\Delta'_3$, and  with  boundary assignments $\Lambda=(\lambda_2,\lambda'_2,\lambda_3,\lambda_3')$. Let $(\Gamma,\mathcal B)$ be a coordinate system of curves of $ X$.
Then the twist parameter $t$  is  determined by $\Lambda$ and $\mlss (X)$. Moreover, we need at most $4$ non-peripheral simple closed curves.
\end{lemma}
\begin{proof}
  It follows immediately from  (\ref{eq:time}).

\end{proof}

\section{Marked length spectral rigidity over $\mathcal{T}_{g,n}$}\label{sec:rigidity}
\subsection{Marked length spectral rigidity about simple closed curves}\label{sec:mlss}
In this section we investigate the relationship between the geometry of a hyperbolic cone surface and its marked length spectrum.
\begin{theorem}\label{thm:mlss}
Let $S$ be a non-exceptionl surface. $\Sim$ is rigid over $\T$. More precisely, let $X_1,X_2\in \T$, if $l_{X_1}([\alpha])=l_{X_2}([\alpha])$ for any $ [\alpha]\in\Sim$, then $X_1=X_2$.
\end{theorem}
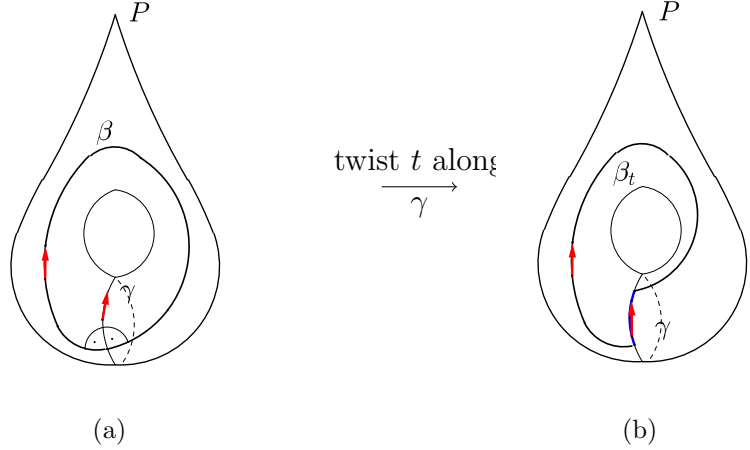
\begin{figure}

    \subfigure[]
    {
     \begin{minipage}[tbp]{50mm}
       \includegraphics[width=50mm]{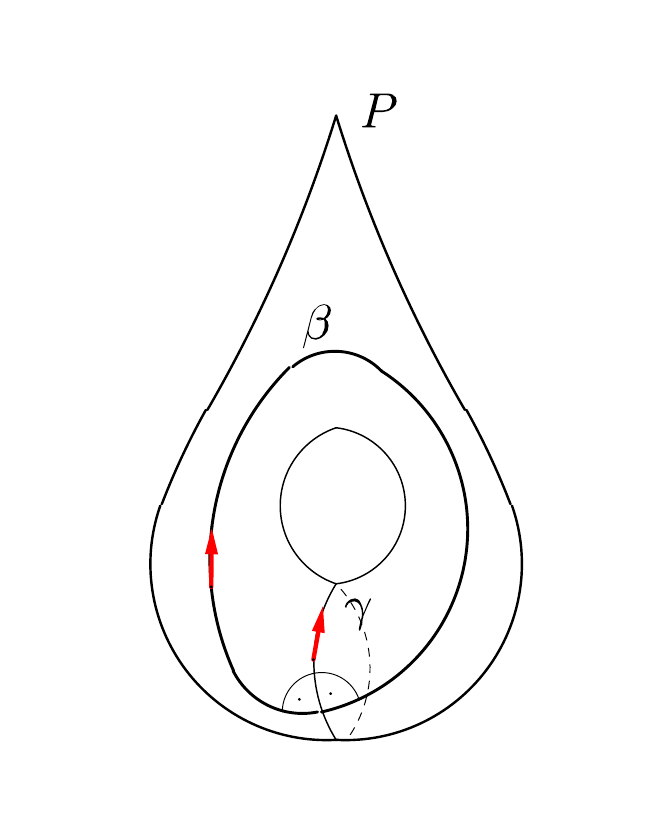}
     \end{minipage}
     \label{fig:turos-cone0}
    }
       \begin{minipage}[tbp]{15mm}
    \begin{tikzpicture}
      \draw [->](-0.5,0)--(0.5,0)
         node[pos=0.5,above] {twist $t$ along }
         node[pos=0.5,below] {$\gamma$};
    \end{tikzpicture}
    \end{minipage}
        \subfigure[]
    {
     \begin{minipage}[tbp]{50mm}
       \includegraphics[width=50mm]{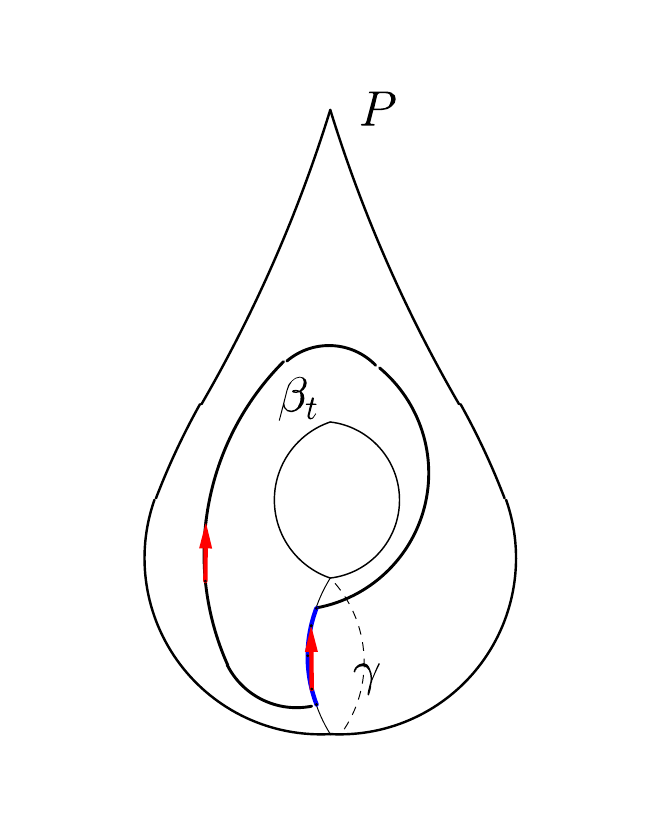}
     \end{minipage}
     \label{fig:turos-conet}
    }
    \label{fig:d&turos}
    \caption{ Figure (a) is a torus with one cone point and with no twist along $\gamma$ with respect to $\beta$, figure (b) is a torus with one cone point and with twist $t$ along $\gamma$ with respect to $\beta$.}
\end{figure}
\begin{proof}
We distinguish two cases.

\vskip 10pt
\textbf{Case } 1.
$S$ is a torus with only one generalized boundary component.
 Let $X$ be such a hyperbolic cone surface  (see Fig.~\ref{fig:turos-cone0} and Fig.~\ref{fig:turos-conet} for the case that the generalized boundary is a cone point). We prove for the case where the generalized boundary is a cone point. The other cases are similar.  First, we choose a non-peripheral simple closed curve $\gamma$ and a collection of seams $\mathcal B=\{\beta,\beta'\}$ where $\beta$ is a non-peripheral simple closed curve which intersects $\gamma$ only once  (see Fig.~\ref{fig:turos-cone0}) and $\beta'$ is a simple arc with endpoints on the generalized boundary component. Then the curve system $(\{\gamma\},\mathcal B)$ defines a Fenchel-Nielsen coordinates 
 of $X$. Let $X_t$ be the hyperbolic cone surface whose Fenchel-Nielsen coordinates are $(\lambda(X),l_X(\gamma),t)$, where $\lambda(X)$ is the cone angle of $X$ and $l_X(\gamma)$ is the length of the geodesic representative of $\gamma$ on $X$.

 To prove the theorem, we need to show that the cone angle and the twist are determined by the lengths of non-peripheral curves. Let $\beta_n$ be the curve obtained from $\beta$ by $n$ times Dehn twist along $\gamma$. Then the twist of $X$ along $\gamma$ with respect to $\beta_n$ is $t+n$. Let $t_{n,l}$ and $t_{n,r}$ be the twist numbers of (geodesic representative of) $\beta_n$ on the left and right side of (the geodesic representative of ) $\gamma$.
 It follows from the definition of twist parameter that $t_{n,l}-t_{n,r}=(t+n)l_X(\gamma)$.

 Cutting $X$ along $\gamma$, we get a pair of pants $R$ where $\gamma$ corresponds to two geodesic boundary components. By the definition of twist parameter, the length $l_{X_0}(\beta)$ of $\beta$ on $X_0$ coincides with the shortest distance between these two geodesic boundary components. As a consequence, the cone angle $\lambda(X)$ of $X$ is determined by $l_{X_0}(\beta)$ and  $l_X(\gamma)$. More precisely,  it follows from  (\ref{eq:Pen-3}) that
  \begin{equation}\label{eq:angle}
    \cos\frac{\lambda(X)}{2}=\sinh^2 \frac{l_X(\gamma)}{2} \cosh l_{X_0}(\beta)-\cosh^2 \frac{l_X(\gamma)}{2}.
  \end{equation}
 Combining the definition of twist parameter and Equation (\ref{eq:dist}), the length $l_X(\beta_n)$ of $\beta_n$ on $X$ is
 $$  l_X(\beta_n)=\cosh^{-1}( \cosh{t_{n,l}}\cosh{t_{n,r}}\cosh l_{X_0}(\beta) -\sinh t_{n,l} \sinh t_{n,r} ).$$
 Hence
 \begin{equation*}
 \begin{array}{cl}
   &\cosh {l_X(\beta_n)}\\
   =&\inf_{s\in \mathbb R} \cosh{s}\cosh{[(s-t-n)l_X(\gamma)]}\cosh l_{X_0}(\beta) -\sinh s \sinh [(s-t-n)l_X)\gamma)] \\
   =&2\cosh\frac{(t+n)l_X(\gamma)}{2}\cosh \frac{l_{X_0}(\beta)}{2}-1.
 \end{array}
 \end{equation*}
 Therefore,
\begin{equation}\label{eq:torus-length}
   \cosh\frac{l_X(\beta_n)}{2}=\cosh\frac{(t+n)l_X(\gamma)}{2}\cosh \frac{l_{X_0}(\beta)}{2}.
\end{equation}
Then
\begin{equation}\label{eq:time2}
\frac{\cosh\frac{l_X(\beta_2)}{2}-\cosh\frac{l_X(\beta_1)}{2}}
{\cosh \frac{l_X(\beta_1)}{2}-\cosh \frac{l_X(\beta)}{2}}=
\frac{\sinh\frac{(2t+3)l_(\gamma)}{4}}
{\sinh\frac{(2t+1)l_X(\gamma)}{4}}.
\end{equation}
From (\ref{eq:angle}), (\ref{eq:torus-length}) and (\ref{eq:time2}),
 the twist parameter $t$ and   the cone angle $\lambda(X)$  can be determined uniquely  by $l_X(\beta_2)$, $l_X(\beta_1)$ $l_X(\beta)$, $l_X(\gamma)$.

\vskip 10pt
\textbf{Case } 2. $S$ is not a torus with only one generalized boundary component.
 In this case, for each generalized boundary component $\Delta_i$, we can find at least one generalized $\xp$ $\mathcal{X}_i$  such that at least two of the generalized boundary components of $\mathcal{X}_i$ are non-peripheral simple closed curves of $S$.
 It follows from Lemma \ref{lem:angsim} that the boundary assignments $\Lambda=(\lambda_1,\lambda_2,...,\lambda_n)$ of $X_1,X_2$ are the same.

Next, we fix a pants decomposition $\Gamma=\{\gamma_i\}_{i=1}^{3g-3+n}$ and  a collection of seams $\mathcal B=\{\beta_1,\cdots,\beta_k\}$.
To prove the theorem, it suffices to show that if $l_{X_1}([\alpha])=l_{X_2}([\alpha])$ for every $ [\alpha]\in\Sim$ then $X_1$ and $X_2$ have the same Fenchel-Nielsen coordinates with respect to $(\Gamma,\mathcal B)$. The pants lengths $L=(l(\gamma_1),l(\gamma_2),...,l(\gamma_{3g-3+n}))$ can be uniquely determined,  and the boundary assignments $\Lambda=(\lambda_1,\lambda_2,...,\lambda_n)$   can also be uniquely determined as above. Finally, by Lemma ~\ref{lem:twist}, the twist parameters $\tw=(t_1,t_2,...,t_{3g-3+n})$ can also be uniquely determined. This completes the proof.

\end{proof}

In fact, the set $\Sim$ is unnecessarily large for the rigidity of $\T$ in the case where $S$ is a non-exceptional surface. From the proof of Lemma \ref{lem:twist} and Lemma~\ref{lem:angsim}, we know that  at most $28n+4(3g-3+n)$ non-peripheral simple closed curves are involved. We summarize this as following:
\begin{theorem}\label{thm:frigidity}
There is a subset $\mathcal{S}'(S)\subset \Sim$ consisting of at most $12g-12+32n$ elements such that $\mathcal{S}'(S)$ is rigid over $\T$ provided that $S$ is a non-exceptionl surface.
\end{theorem}

 Now Theorem A  follows from Theorem \ref{thm:mlss} and Theorem \ref{thm:frigidity}.
\remark In the case of closed surface, the subset $\mathcal{S}'(S)$ can be chosen such that it has $6g-5$ elements \cite{Sch}.

If we consider the subspace $\mathcal{T}_{g,n}(\Lambda)$ of $\mathcal{T}_{g,n}$, we get the following result.
\begin{theorem}\label{thm:mlssb}
Suppose $\Lambda\in(-\pi,\infty)^n$ and that  $S$ is a non-exceptional surface. Then $\Sim$ is rigid over $\mathcal{T}_{g,n}(\Lambda)$, i.e. if $l_{X_1}([\alpha])=l_{X_2}([\alpha])$ holds for any $ [\alpha]\in\Sim$, then $X_1=X_2$.
\end{theorem}
\begin{proof}
  The proof is similar to that of Theorem \ref{thm:mlss} except that we do not need to deal with the boundary assignments here.

\end{proof}

\subsection{Generalized marked length spectral rigidity}\label{sec:mlssn}

In~\cite{TWZ}, the authors proved a generalized McShane's identity  for length series of simple closed geodesics on a hyperbolic cone surface by studying ``gaps'' formed by simple-normal geodesics emanating from a distinguished cone point, cusp or boundary geodesic. For each pair of pants with three generalized boundary components $\Delta_0,\alpha,\beta$, where $\Delta_0$ is the distinguished generalized boundary component, they define a  {\it Gap function} $Gap(\Delta_0;\alpha,\beta)$ (see \cite{TWZ} for the definition of $Gap(\Delta_0;\alpha,\beta)$) . 

In the theorem below, a generalized simple closed geodesic is a simple closed geodesic,  a cone point, or a cusp.
\begin{theorem}[\cite{TWZ}]\label{thm:twz}
Let $X$ be a hyperbolic cone-surface with all cone angles
in $[0, \pi)$, and ${\Delta}_{0}$ a distinguished generalized boundary component.  Then one has either
\begin{eqnarray}\label{eqn:001}
\sum {\rm Gap}({\Delta}_{0};\alpha, \beta) = \frac{\theta_0}{2},
\end{eqnarray}
when ${\Delta}_{0}$ is a cone point of cone angle $\theta_{0}$;
\quad or
\begin{eqnarray}\label{eqn:002}
\sum {\rm Gap}({\Delta}_{0};\alpha, \beta) = \frac{l_{0}}{2},
\end{eqnarray}
when ${\Delta}_{0}$ is a boundary geodesic of length $l_{0}$; \quad
or
\begin{eqnarray}\label{eqn:00'}
\sum {\rm Gap}({\Delta}_{0};\alpha, \beta) = \frac{1}{2},
\end{eqnarray}
when ${\Delta}_{0}$ is a cusp; where in each case the sum is over
all pairs of generalized simple closed geodesics $\alpha, \beta$
on $M$ which bound with ${\Delta}_{0}$ an embedded pair of pants.
\end{theorem}
\remark \label{appear} For any  non-peripheral simple closed  geodesic $\gamma$, there is a generalized simple closed geodesic $\gamma'$ with $\gamma'\neq \gamma$, such that $\gamma,\gamma'$ and $\Delta_0$ bounds an embedded pair of pants. Hence $\gamma$ appears  in the left side of each identity (see Fig. ~\ref{fig:pants-decomp}).

\vskip 5pt
Combining Theorem \ref{thm:mlss} and Theorem \ref{thm:twz}, we get the following generalized marked length spectral rigidity.
\begin{theoremc}\label{cor:mlssn}
Suppose that $\Lambda\in(-\pi,\infty)^n$ and that  $S$ is a non-exceptional surface. Let  $X_1,\ X_2\in \mathcal{T}_{g,n}(\Lambda)$. If $l_{X_1}([\alpha])\geq l_{X_2}([\alpha])$ holds for every $[\alpha] \in\Sim$, then $X_1=X_2$.
\end{theoremc}
\begin{proof}
First, we take a cone point $P_0$  as  the distinguished cone point $\Delta_0$. If there is no cone point, we pick a cusp or a geodesic boundary component  instead. Since  $X_1,\ X_2\in \mathcal{T}_{g,n}(\Lambda)$, they share the same cone angles and  boundary lengths, which lead to
\begin{equation}\label{twz-id}
 \sum {\rm Gap}_{X_1}(\Delta_{0};\alpha, \beta) = \frac{\theta}{2}=\sum {\rm Gap}_{X_2}(\Delta_{0};\alpha, \beta),
\end{equation}
where  the sum is over
all pairs of generalized simple closed geodesics $\alpha, \beta$
on $S$ which bound with $P_{0}$ an embedded pair of pants.

Note that  each \textit{gap function} is a strictly decreasing function about the length of each involved  non-peripheral simple closed geodesic. Combining this with the assumption that $l_{X_1}(\alpha)\geq l_{X_2}(\alpha),$  for any $ \alpha \in\Sim$, we have
\begin{equation}\label{ineq}
 \sum {\rm Gap}_{X_1}(\Delta_{0};\alpha, \beta) \leq \sum {\rm Gap}_{X_2}(\Delta_{0};\alpha, \beta),
\end{equation}

\noindent with the equality holds if and only if $l_{X_1}(\alpha)=  l_{X_2}(\alpha)$ for any $\alpha \in\Sim$ (by Remark~\ref{appear} every $\alpha \in\Sim$ is involved in the sum). Combining with (\ref{twz-id}) and  (\ref{ineq}), we have $l_{X_1}(\alpha)=  l_{X_2}(\alpha)$ for any $ \alpha \in\Sim$.

Now, the theorem follows directly from Theorem~\ref{thm:mlssb}.
\end{proof}

As an application, we can define the Thurson's asymmetric metric on $\mathcal{T}_{g,n}(\Lambda)$ for any $\Lambda\in(-\pi,\infty)^n$.
\begin{corollary}[the Thurston metric]\label{thmetric}
   Suppose  $\Lambda\in(-\pi,\infty)^n$ and that   $S$ is a non-exceptional surface.  For any two points $X_1,X_2\in \mathcal{T}_{g,n}(\Lambda)$,
$$d_{Th}(X_1,X_2)\triangleq \log \sup_{[\alpha]\in\Sim}\frac{l_{X_2}([\alpha])}{l_{X_1}([\alpha])} $$
defines an asymmetric metric on $\mathcal{T}_{g,n}(\Lambda)$, which is called the Thurston metric.
\end{corollary}

\section{Comparisons of geometry between hyperbolic cone surfaces}\label{sec:comp-Y}

In this section, we build a connection between a general hyperbolic cone surface and a punctured hyperbolic surface by comparing the lengths of the geodesic representatives of  isotopy classes of non-peripheral simple closed curves on both surfaces. Before that, we prove some lemmas first.
\begin{lemma}\label{lemma:comparisons}
    Assume $c,c'>0$ and let $d(c,\rho_1,\rho_2),f(c,\rho_1,\rho_2)$ be two functions of $\rho_1,\rho_2$ and $c$ defined by
  $$ \cosh d(c,\rho_1,\rho_2)=f(c,\rho_1,\rho_2)\triangleq \cosh{\rho_1}\cosh{\rho_2}\cosh c -\sinh \rho_1 \sinh \rho_2.$$
  If
   $$  \frac{1}{C}\leq\frac{\cosh c-1}{\cosh c'-1}\leq C $$
  for some $C>1.$ Then
    \begin{enumerate}
    [(a)]
    \item 
        $$ |d(c',\rho_1,\rho_2)-d(c,\rho_1,\rho_2)|
       \leq  \textup{arccosh }C.$$
    \item 
        $$\frac{1}{C}\leq \frac{d(c,\rho_1,\rho_2)}
        {d(c',\rho_1,\rho_2)}\leq C.$$
  \end{enumerate}
\end{lemma}
\begin{proof}
  Part (a) follows from Lemma~\ref{lemma:comparison-1}(a) and  \ref{lemma:comparison-2}(a), part (b) follows from Lemma~\ref{lemma:comparison-1}(b) and  \ref{lemma:comparison-2}(b).
\end{proof}

\begin{lemma}\label{lemma:comparison-1}
Let $x,y>0$.
\begin{enumerate}
  [(a)]
  \item  If
            $$ \frac{1}{C}\leq\frac{\cosh x}{\cosh y}\leq C,$$
         then
            $$|x-y|\leq \textup{arccosh } C.$$
  \item  If
            $$ \frac{1}{C}\leq\frac{\sinh x}{\sinh y}\leq C,$$
         then
            $$\frac{1}{C}\leq\frac{x}{y}\leq C.$$
\end{enumerate}
\end{lemma}
\begin{proof}
   (a) Let $F(y)$ be defined by
   $$ \cosh(y+F)=C\cosh y.$$
   Hence
  $$ \frac{e^{y+F}+e^{-y-F}}{2}=C\cosh y ,$$
  then
  $$
  F(y)={\ln (C\cosh y+\sqrt{C^2\cosh^2y-1})}-{y},
  $$
  further,
  $$
  F'(y)=\frac{C\sinh y}{\sqrt{C^2\cosh^2y-1}}-1<0.
  $$
  As a result,
  $$F(y)\leq F(0)=\textup{arccosh }C.$$
  On the other hand,
  $$\cosh x\leq C\cosh y=\cosh (y+F(y)). $$
  It follows that
   $$x-y\leq \textup{arccosh }C.$$
   By symmetry, we have $y-x\leq \textup{arccosh }C.$
   \vskip 5pt
  (b)
  Let $K(y)$ be a function of $y$ defined by
  $$ \sinh (yK)=C\sinh y.$$
  Hence
  $$ \frac{e^{yK}-e^{-yK}}{2}=C\sinh y ,$$
  then
  $$
  K(y)=\frac{\ln (C\sinh y+\sqrt{1+C^2\sinh^2y})}{y},
  $$
  further,
  $$
  K'(y)=\frac{C\cosh y}{y\sqrt{1+C^2\sinh^2y}}-\frac{{\ln (C\sinh y+\sqrt{1+C^2\sinh^2y})}}{y^2}.
  $$
  Setting
  $K'(y_0)=0$,
  we have
  \begin{eqnarray*}
     K(y_0)
     &=&\frac{\ln (C\sinh y_0+\sqrt{1+C^2\sinh^2y_0})}{y_0}\\
     &=&\frac{C\cosh y_0}{\sqrt{1+C^2\sinh^2y_0}}\\
     &<&C.
  \end{eqnarray*}
 Note that
 \begin{eqnarray*}
   K(y)\to C,\ \textup{as }y\to0 &\textup{and}&  K(y)\to 1,\ \textup{as }y\to\infty.
 \end{eqnarray*}
 Therefore $K(y)\leq C$ for any $y\geq 0$.   Since $\sinh x\leq C \sinh y =\sinh yK(y)$, $x/y\leq C$.  By symmetry, we also have $y/x\leq C.$

\end{proof}

   \begin{lemma}\label{lemma:comparison-2}
  Assume $0<c\leq c'$ and let $f(c,\rho_1,\rho_2)$ be the function defined in Lemma~\ref{lemma:comparisons}.
  then
  \begin{enumerate}
    [(a)]
    \item
    $$\frac{f(c,\rho_1,\rho_2)}{f(c',\rho_1,\rho_2)}
    \geq \frac{\cosh c-1}{\cosh c'-1}.$$
    \item
    $$\frac{f(c,\rho_1,\rho_2)^2-1}{f(c',\rho_1,\rho_2)^2-1}
    \geq ( \frac{\cosh c-1}{\cosh c'-1})^2.$$
  \end{enumerate}

 \end{lemma}
 \begin{proof}
 (a) Note that
 \begin{eqnarray*}
    \frac{f(c,\rho_1,\rho_2)}{f(c',\rho_1,\rho_2)}
    &=& \frac{\cosh{\rho_1}\cosh{\rho_2}\cosh c -\sinh \rho_1 \sinh \rho_2}{\cosh{\rho_1}\cosh{\rho_2}\cosh c' -\sinh \rho_1 \sinh \rho_2}\\
    &=& 1-\frac{\cosh c'-\cosh c}{\cosh c'-\tanh \rho_1\tanh \rho_2}\\
    &\geq& 1-\frac{\cosh c'-\cosh c}{\cosh c'-1}\\
    &=&\frac{\cosh c-1}{\cosh c'-1}.
 \end{eqnarray*}
   (b) Observe that
        \begin{eqnarray*}
          \frac{f(c,\rho_1,\rho_2)^2-1}{f(c',\rho_1,\rho_2)^2-1}
          &  =  & \frac{f(c,\rho_1,\rho_2)+1}{f(c',\rho_1,\rho_2)+1}
                  \frac{f(c,\rho_1,\rho_2)-1}{f(c',\rho_1,\rho_2)-1}\\
          &\geq & \frac{f(c,\rho_1,\rho_2)}{f(c',\rho_1,\rho_2)}
                  \frac{f(c,\rho_1,\rho_2)-1}{f(c',\rho_1,\rho_2)-1},\\
        \end{eqnarray*}
    and
     \begin{eqnarray*}
     \frac{f(c,\rho_1,\rho_2)-1}{f(c',\rho_1,\rho_2)-1}
     &=& 1-\frac{\cosh c'-\cosh c}{\cosh c'-\frac{\sinh \rho_1\sinh \rho_2+1}{\cosh \rho_1\cosh\rho_2}}\\
    &\geq& 1-\frac{\cosh c'-\cosh c}{\cosh c'-1}\\
    && (\textup{since }\cosh(\rho_1-\rho_2)\geq1)\\
    &=&\frac{\cosh c-1}{\cosh c'-1}.
   \end{eqnarray*}
    Then we have
    $$\frac{f(c,\rho_1,\rho_2)^2-1}{f(c',\rho_1,\rho_2)^2-1}
    \geq ( \frac{\cosh c-1}{\cosh c'-1})^2.$$

 \end{proof}

 Recall that  a generalized boundary component $\Delta$ is a cone point with cone angle $\theta\in (0,\pi)$,  a cusp, or a geodesic boundary of positive length. The assignment  $\lambda (\Delta)$ at a generalized boundary component is:
  \begin{equation}
    \lambda(\Delta)=\left \{
                            \begin{array}{ll}
                              -\theta, & \text{ if } \Delta \text{ is a cone point of angle } \theta\in(0,\pi),\\
                              0,& \text{ if } \Delta \text{ is a cone point},\\
                              l,& \text{ if } \Delta \text{ is a geodesic boundary of length } l>0.
                            \end{array}
                    \right.
  \end{equation}

 Let $ V(\lambda_1,\lambda_2,\lambda_3)$  be a \textit{ generalised Y-piece}
 whose  boundary assignments satisfy $\Lambda=(\lambda_1,\lambda_2,\lambda_3)\in (0,\infty)\times (-\pi, \infty)\times(-\pi,\infty)$, and $\Delta_1,\Delta_2,\Delta_3$  be the three  marked generalized boundary components. Let $\zeta$  be an arbitrary simple arc with  endpoints  $M,N$ on the  geodesic boundary components (not a cone point nor a cusp). Assume that  $M$ is on $ \Delta_1$ and $N$ is on $\Delta^N\in\{\Delta_1,\Delta_2,\Delta_3\}$. Denote by $g_\zeta$ the geodesic representative in the corresponding homotopy class of $\zeta$ relative to the endpoints, i.e. the endpoints $M,N$ stay fixed during the homotopy, and  by $h_{\zeta}$ the geodesic representative in the homotopy class of $\zeta$ relative to the geodesic boundaries of $ V(\lambda_1,\lambda_2,\lambda_3)$, i.e. the endpoints $M,N$ stay at the  geodesic boundaries  during the homopoty process.
 Let $l(\lambda_1,\lambda_2,\lambda_3,[\zeta])$ be the length of $g_\zeta$.
 Let $h_{\Lambda}$ (resp. $h_0$) be  the length of $h_\zeta$ on $ V(\lambda_1,\lambda_2,\lambda_3)$ (resp. $ V(\lambda_1,\lambda_2,0)$ or $ V(\lambda_1,0,0)$, up to the situation in consideration).

 Let the geodesic boundary $\Delta_1,\Delta^N$ be oriented such that $ V(\lambda_1,\lambda_2,\lambda_3)$ sits on the right of $\Delta_1$ and on the left of $\Delta^N$. Since $g_\zeta$ and $h_\zeta$ are homotopic relative to the boundaries of $V(\lambda_1,\lambda_2,\lambda_3)$, there  is a homotopy $H$ between $g_\zeta$ and $h_\zeta$ on $V(\lambda_1,\lambda_2,\lambda_3)$. Denote by $\rho_M$ and $\rho_N$ be the signed displacements of $h_\zeta\cap \Delta_1$ and $h_\zeta\cap\Delta^N$ during the homotopy process.

 \remark The assumption $\Lambda=(\lambda_1,\lambda_2,\lambda_3)\in (0,\infty)\times (-\pi, \infty)\times(-\pi,\infty)$ guarantees  that $\Delta_1$ is a geodesic boundary. The reason for this assumption is that the surface we are interested  is not a sphere with three holes, which means that each pair of pants on the surface has at least one boundary component which is an non-peripheral simple closed curve.

\begin{figure}
  \subfigure[]
    {
    \begin{minipage}[tbp]{50mm}
      \includegraphics[width=50mm]{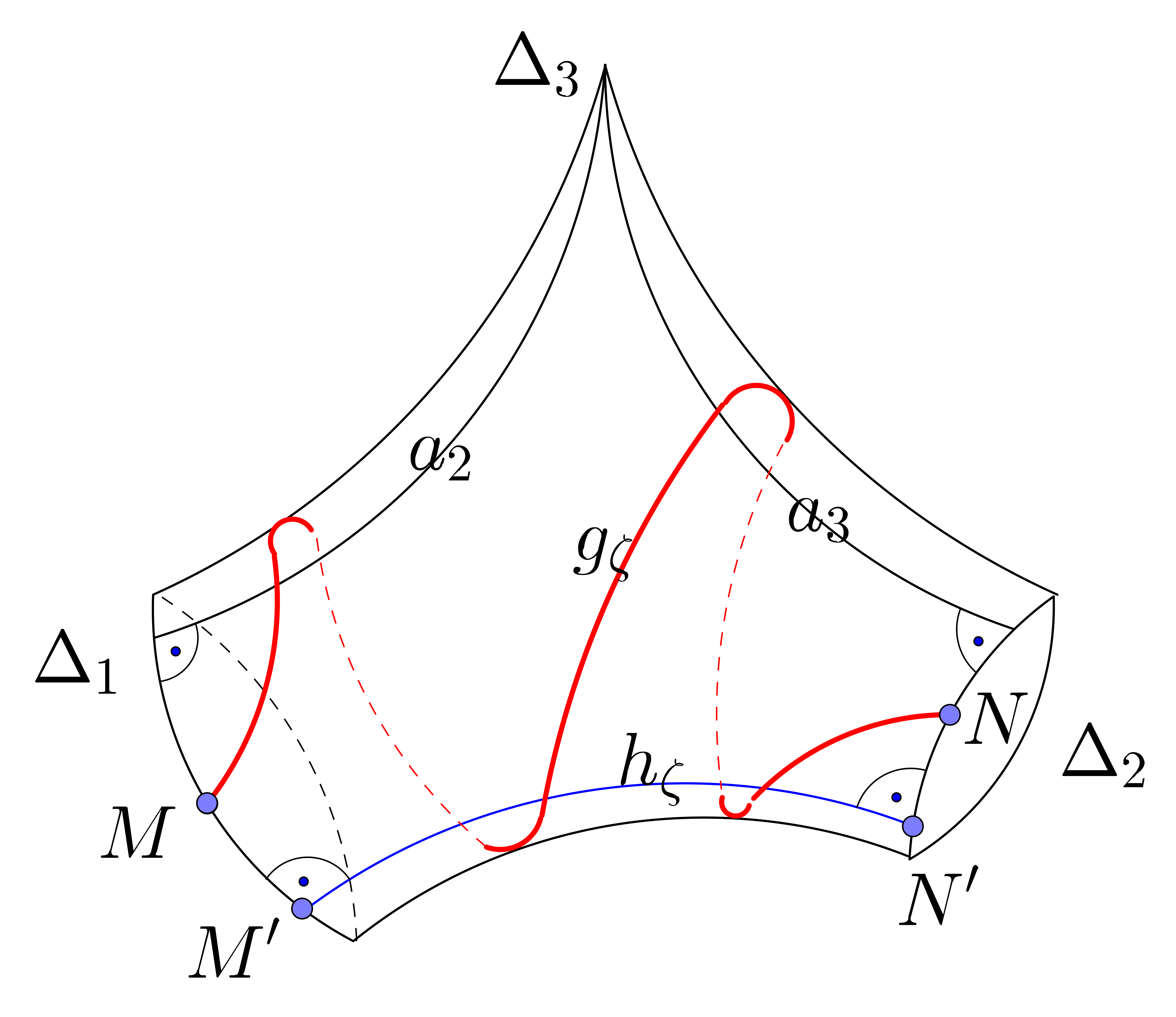}
    \label{fig:Vcompare1}
    \end{minipage}
    }
  \subfigure[]
  {
    \begin{minipage}[tbp]{50mm}
      \includegraphics[width=50mm]{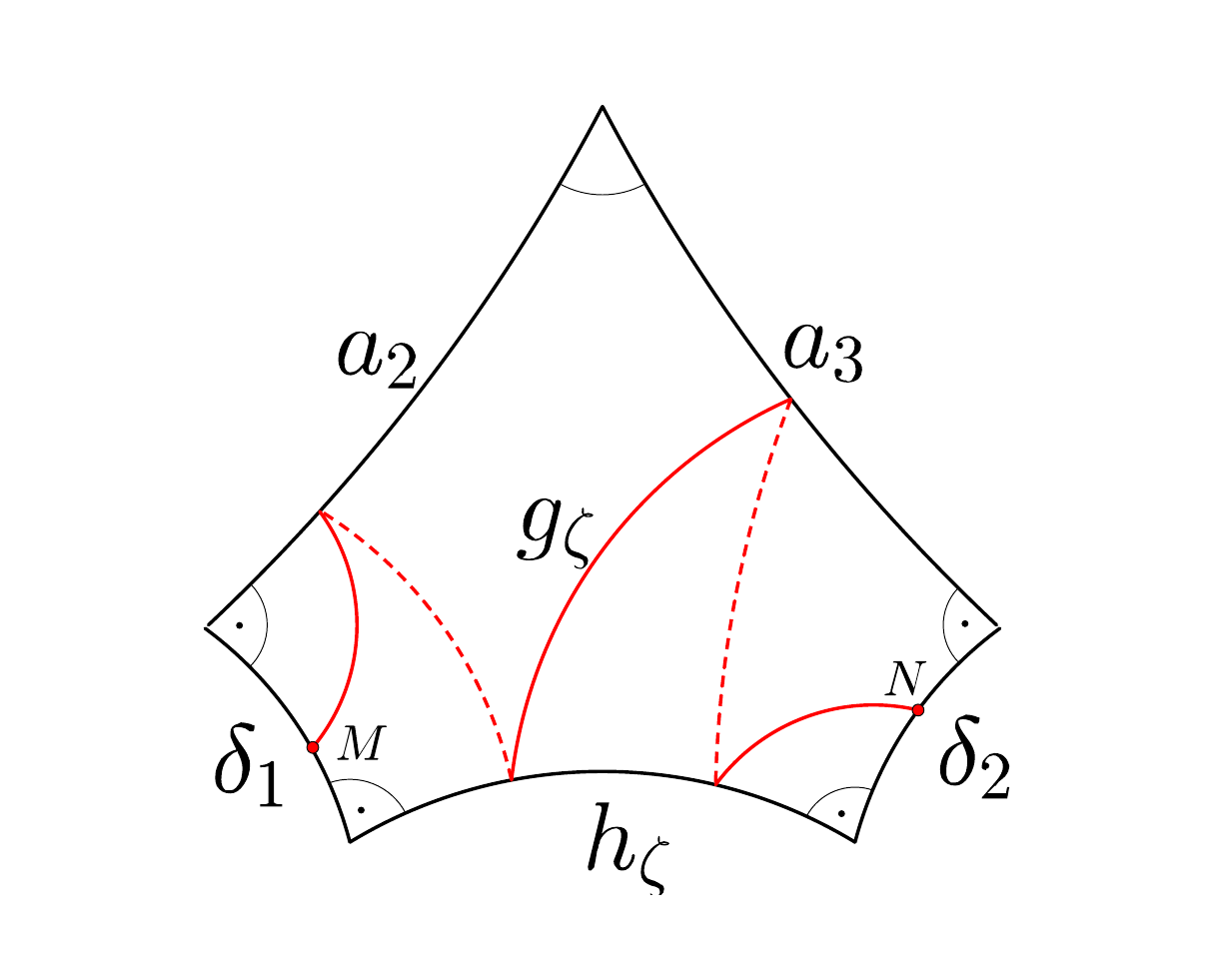}
    \label{fig:V-pen-1}
    \end{minipage}
    }

  \subfigure[]
  {
  \begin{minipage}[tbp]{50mm}
  \includegraphics[width=50mm]{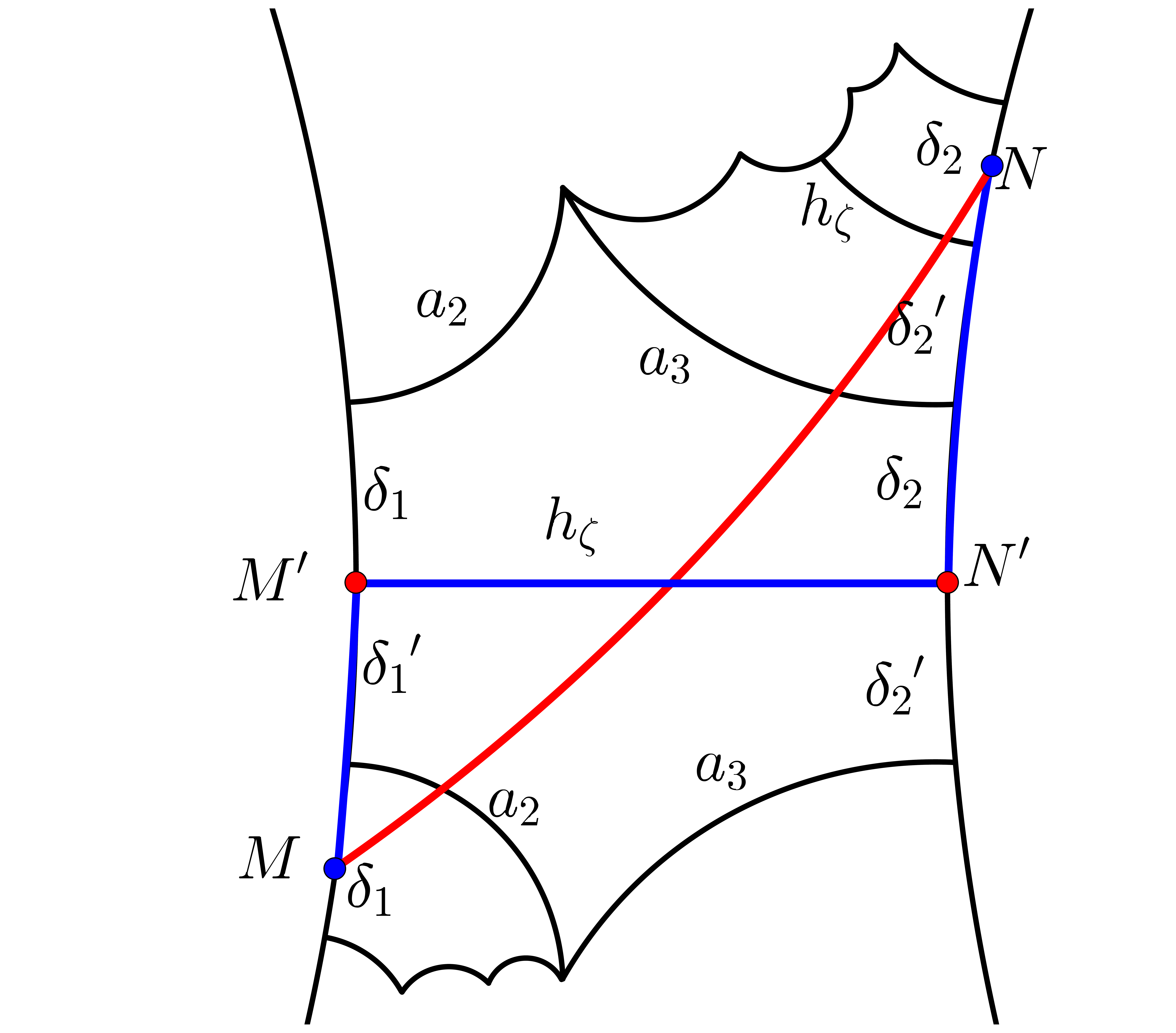}
  \label{fig:eta1}
 \end{minipage}
  }
  \caption{Figure (c) pictures the extension of $g_\zeta$ to the hyperbolic plane. The V-piece in Figure (a) consists of two isometric pentagons in Figure (b). If we record the edges of these two pentagons meet by $g_\zeta$ during the extension, the geodesic $g_\zeta$ in Figure (a) corresponds to a sequence $(\delta_1,a_2,h_\zeta,a_3,h_\zeta,\delta_2)$.}
\end{figure}
\begin{lemma}\label{lem:eta}
  With notations above, we have
  $$           \cosh l(\lambda_1,\lambda_2,\lambda_3,[\zeta])= \cosh{\rho_M}\cosh{\rho_N}\cosh h_\Lambda -\sinh \rho_M \sinh \rho_N.
 $$
\end{lemma}
\begin{proof}
 We consider the case that $\Lambda=(\lambda_1,\lambda_2,\lambda_3)\in (0,\infty)\times (0, \infty)\times(-\pi,\infty)$, the proofs for the remaining cases are similar.

 Recall that $V(\lambda_1,\lambda_2,\lambda_3)$ consists of two isometric pentagons pictured in Figure \ref{fig:V-pen-1}. 
 Extend the geodesic $g_\zeta$ on $V(\lambda_1,\lambda_2,\lambda_3)$  to the hyperbolic plane as illustrated in Fig.\ref{fig:eta1}. Then the lemma follows from equation (\ref{eq:dist}) for the quadrilateral.
\end{proof}
 The lemma below describes the comparisons of the geometry between a pair of general pants  and a pair of pants which do not contain any cone point.
\begin{lemma}\label{lem:Y-compare}
    With the notations above, we have the following estimates:

  \begin{enumerate}
    [(a)]
    \item If $\Delta_1,\Delta_2$ are geodesic boundaries, $\Delta_3$ is a generalized geodesic boundary, then  there exist $ C_1,D_1>0$, which depend only on $\lambda_3$, such that
        $$ |[l(\lambda_1,\lambda_2,\lambda_3,[\zeta])
        -l(\lambda_1,\lambda_2,0,[\zeta]))|
       \leq D_1,$$
       and
       $$\frac{1}{C_1}\leq \frac{l(\lambda_1,\lambda_2,\lambda_3,[\zeta])}
        {l(\lambda_1,\lambda_2,0,[\zeta])}\leq C_1.$$
        Moreover, $D_1\to0,C_1\to1$ as $\lambda_3\to0$.
    \item If $\Delta_1$ is a  geodesic boundary, $\Delta_2$ and $\Delta_3$ are generalized geodesic boundaries with the same type, then there exist $ C_2,D_2>0$, which depend only on $\lambda_2,\lambda_3$, such that
        $$ |l(\lambda_1,\lambda_2,\lambda_3,
        [\zeta])  -l(\lambda_1,0,0,[\zeta])|
       \leq D_2,$$
       and
       $$\frac{1}{C_2}\leq \frac{l(\lambda_1,\lambda_2,\lambda_3,[\zeta])}
        {l(\lambda_1,0,0,[\zeta])}
        \leq {C_2}.$$
        Moreover, $D_2\to0,C_2\to1$ as $\lambda_2,\lambda_3\to0$.
    \item If $\Delta_1$ is a  geodesic boundary, $\Delta_2$ and $\Delta_3$ are generalized geodesic boundaries with different types,  then there exist $ C_3,D_3>0$, which depend only on $\lambda_2,\lambda_3$, such that
        $$  |l(\lambda_1,\lambda_2,\lambda_3,[\zeta]) -l(\lambda_1,0,0,[\zeta])|
       \leq D_3,$$
       and
       $$\frac{1}{C_3}\leq \frac{l(\lambda_1,\lambda_2,\lambda_3,[\zeta])}
        {l(\lambda_1,0,0,[\zeta])}\leq {C_3}.$$
        Moreover, $D_3\to0,C_3\to1$, as $\lambda_2\to0,\lambda_3\to0$.
  \end{enumerate}
\end{lemma}

\remark
\begin{enumerate}[(i)]
  \item Here the  point is that the estimates 
      have nothing to do with the homotopy classes of $\zeta$.
  \item The estimates about the difference of corresponding lengths works for the situation where the lengths tend to infinity, while the estimates about the ratio of corresponding lengths works for the situation where the lengths tend to zero. Both of them will be used to prove the quasi-isometry of $\mathcal{T}_{g,n}(\Lambda)$ and  $\mathcal{T}_{g,n}(0)$  and to solve various boundary problems.
\end{enumerate}

\begin{proof}
  \begin{figure}
            \subfigure[]
    {
    \begin{minipage}[tbp]{50mm}
      \includegraphics[width=50mm]{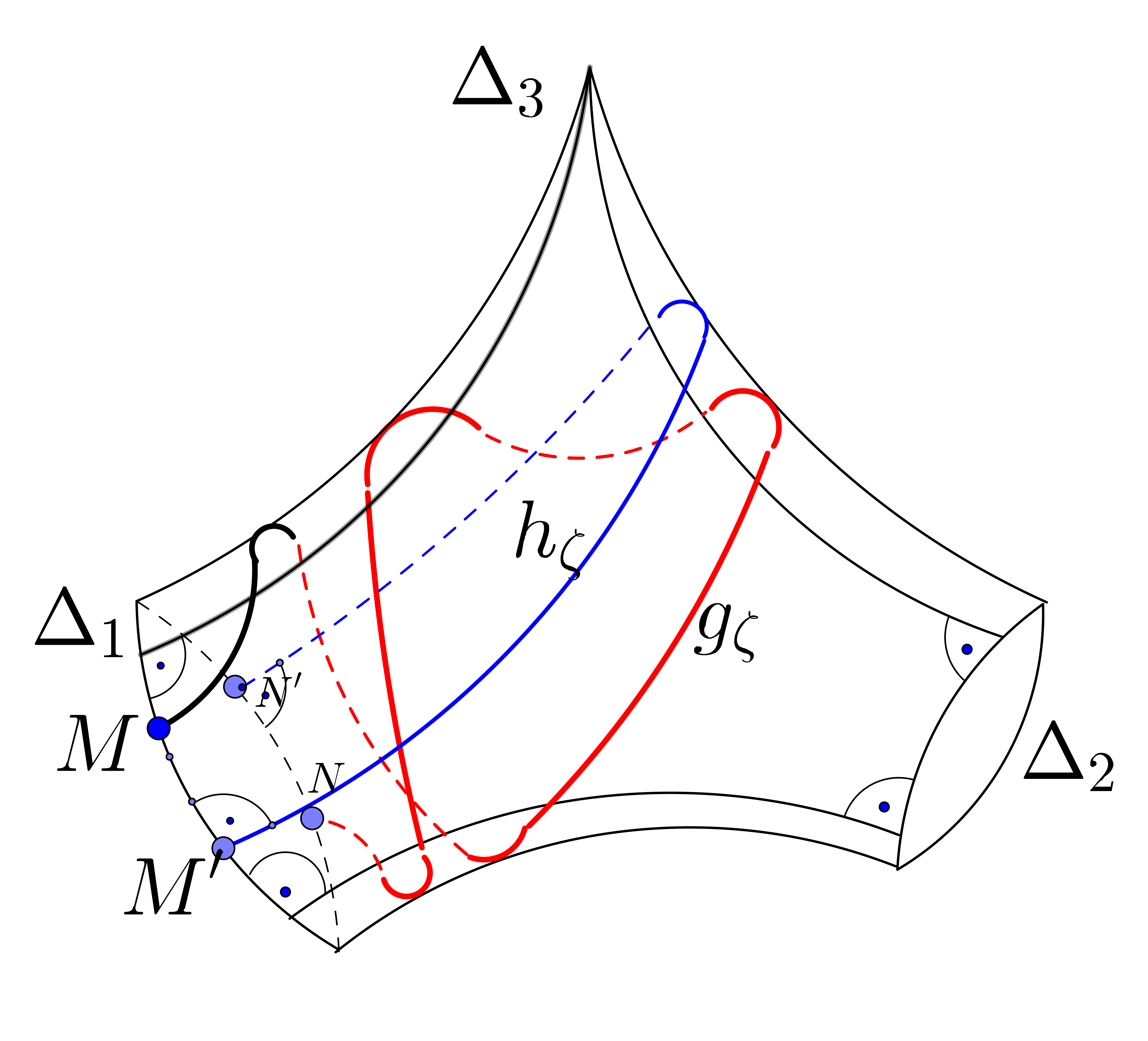}
    \label{fig:Vcompare2}
    \end{minipage}
    \begin{minipage}[tbp]{50mm}
      \includegraphics[width=50mm]{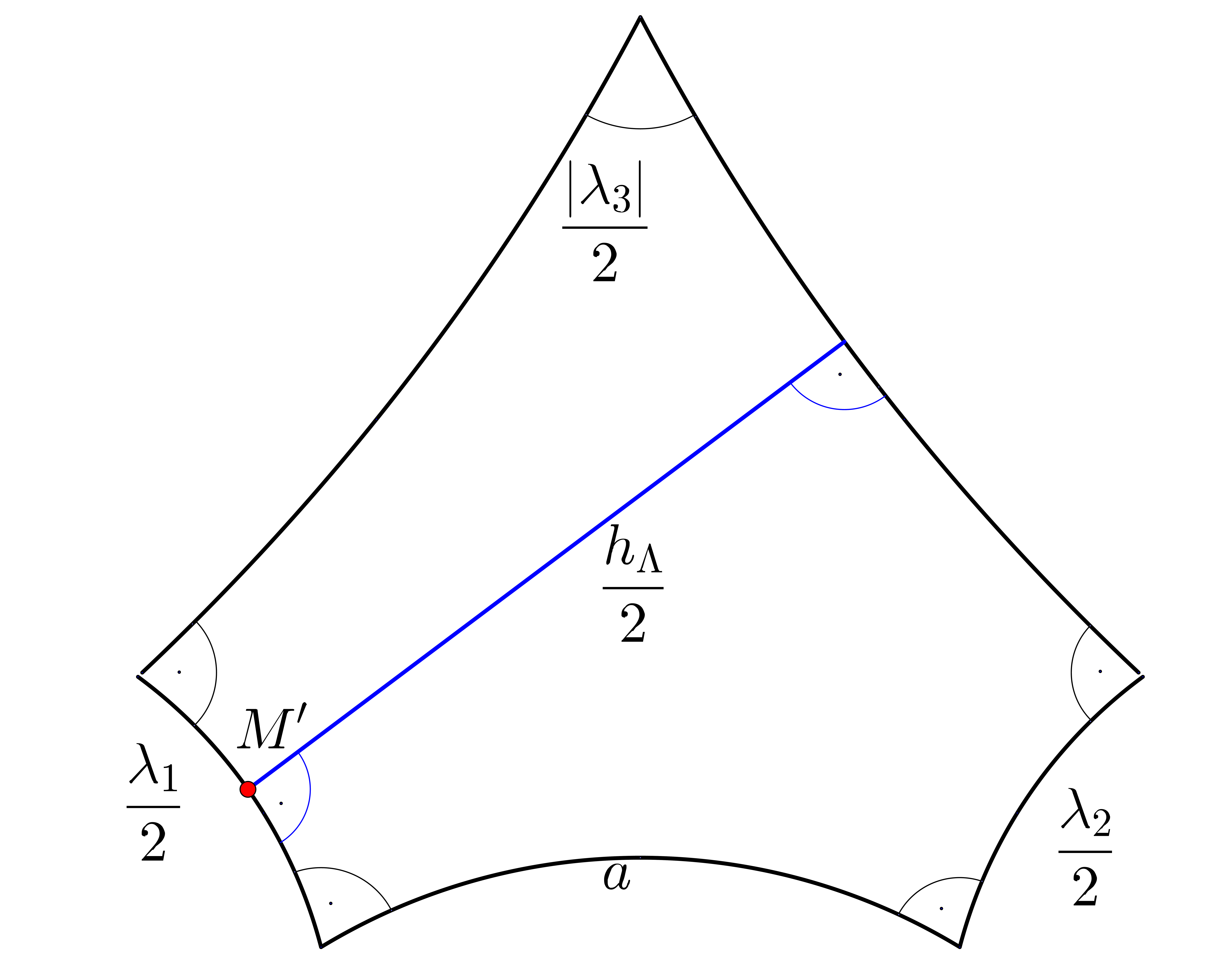}
    \label{fig:V-pen-2}
    \end{minipage}
    }

        \subfigure[]
    {
    \begin{minipage}[tbp]{50mm}
      \includegraphics[width=50mm]{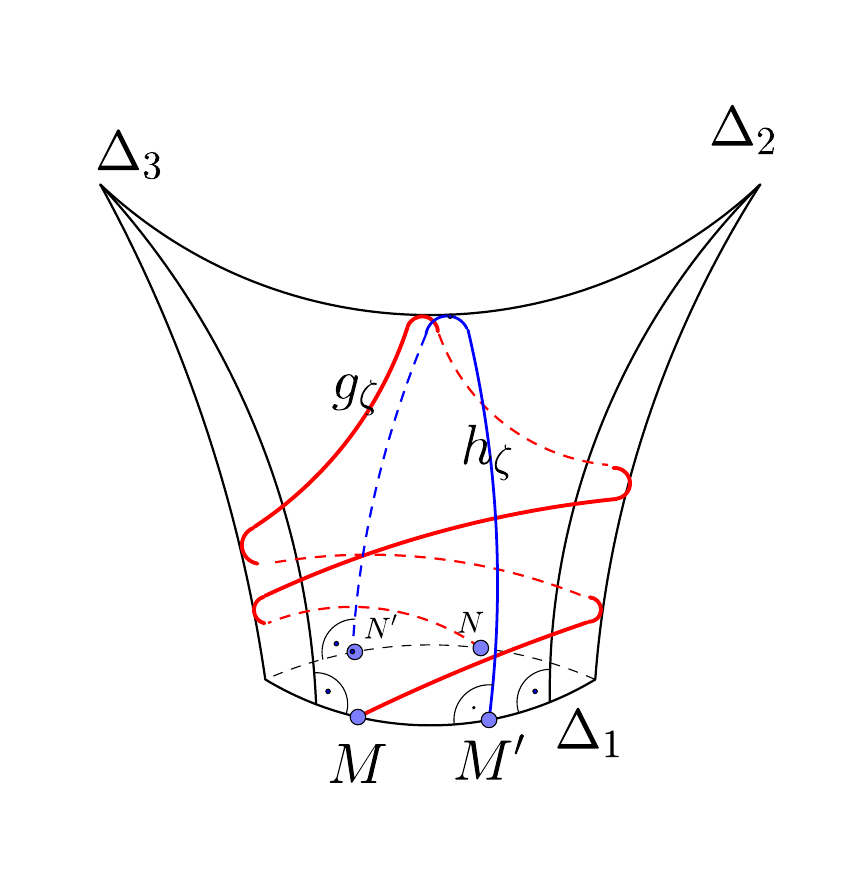}
    \label{fig:Jcompare}
    \end{minipage}
    \begin{minipage}[tbp]{50mm}
      \includegraphics[width=50mm]{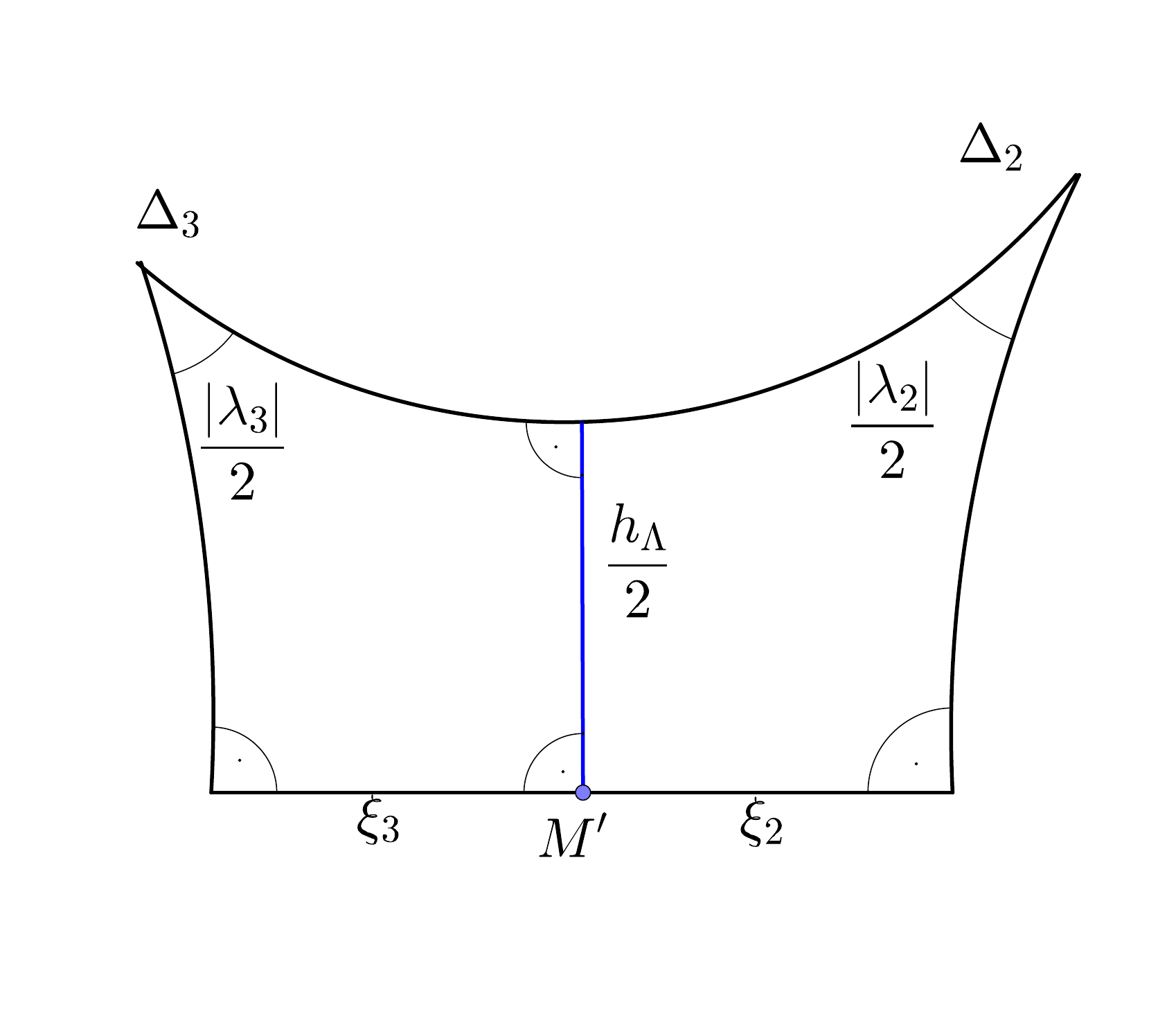}
    \label{fig:J-Quadri}
    \end{minipage}
    }
    \caption{The arc $\zeta$ in the  figures on the left are arbitrary geodesic arcs with endpoint on given boundary components, and $h_{\zeta}$ is the geodesic homotopic to $\zeta$ relative to the boundary components of $V$. Each figure on the right side is one of the two isometric polygons consisting the corresponding pair of pants on its left. }
    \label{fig:comp-Y-pieces}
  \end{figure}
 For simplicity, we  denote by $h_{\Lambda}$ (resp. $h_0$) the length of $h_\zeta$ on $ V(\lambda_1,\lambda_2,\lambda_3)$ (resp. $ V(\lambda_1,\lambda_2,0)$ or $ V(\lambda_1,0,0)$, up to the situation in consideration).

It follows from Lemma \ref{lem:eta} that
$$\left\{
\begin{array}{rcl}
  \cosh l(\lambda_1,\lambda_2,\lambda_3,[\zeta])&= & \cosh{\rho_M}\cosh{\rho_N}\cosh h_\Lambda -\sinh \rho_M \sinh \rho_N\\

 \cosh l(\lambda_1,\lambda_2,0,[\zeta])&= &\cosh{\rho_M}\cosh{\rho_N}\cosh h_0 -\sinh \rho_M \sinh \rho_N\\
\end{array}
\right.
$$ for statement (a), and
$$\left\{
\begin{array}{rcl}
  \cosh l(\lambda_1,\lambda_2,\lambda_3,[\zeta])&= & \cosh{\rho_M}\cosh{\rho_N}\cosh h_\Lambda -\sinh \rho_M \sinh \rho_N\\

 \cosh l(\lambda_1,0,0,[\zeta])&= &\cosh{\rho_M}\cosh{\rho_N}\cosh h_0 -\sinh \rho_M \sinh \rho_N\\
\end{array}
\right.
$$ for statement (b),(c).

  By Lemma~\ref{lemma:comparisons},  it suffices to prove that
   $$ \frac{1}{K} \leq\frac{\cosh h_{\Lambda}-1}{\cosh h_0-1}\leq K$$
for some $K$ that approaches 1 as  we take the limit for each statement in the lemma.

 We distinguish  six cases. The statement (a) is obtained from the cases A-D, and the statement (b) is obtained from the cases E and F. By statement (a) and statement (b), we have statement (c).

\vskip 10pt
\textbf {Case} A.  $\Delta_1,\Delta_2$ are geodesic boundaries, $\Delta_3$ is a cone point or a cusp and $N\in\Delta_2$ (see Fig.~\ref{fig:Vcompare1}).
\vskip 5pt
Using equation (\ref{eq:Pen-3}) for the pentagon pictured in Fig.~\ref{fig:V-pen-1}, we have
$$ \cosh h_{\Lambda}=\frac{\cosh \frac{\lambda_1}{2}\cosh \frac{\lambda_2}{2}+\cos \frac{|\lambda_3|}{2}}{\sinh \frac{\lambda_1}{2}\sinh \frac{\lambda_2}{2}},$$
where ${\lambda_1}$ and ${\lambda_2}$ are the lengths of the geodesic boundaries $\Delta_1$ and $\Delta_2$ respectively, $|\lambda_3|$ is the cone angle of $\Delta_3$.
Hence
\begin{eqnarray*}
  \frac{\cosh h_{\Lambda}-1}{\cosh h_0-1}=\frac{\cosh\frac{\lambda_1-\lambda_2}{2}+\cos\frac{|\lambda_3|}{2}
  }{\cosh\frac{\lambda_1-\lambda_2}{2}+1}
  \in[\frac{1+\cos\frac{|\lambda_3|}{2}}{2},1].
\end{eqnarray*}

\vskip 20pt

\textbf {Case} B. $\Delta_1,\Delta_2$ are geodesic boundaries, $\Delta_3$ is a geodesic boundary  and $N\in\Delta_2$.
\vskip 5pt
Replace  $\cos \frac{|\lambda_3|}{2}$ in {case} A  by $\cosh \frac{|\lambda_3|}{2}$, we have
\begin{eqnarray*}
  \frac{\cosh h_{\Lambda}-1}{\cosh h_0-1} &=&\frac{\cosh\frac{\lambda_1-\lambda_2}{2}+\cosh\frac{|\lambda_3|}{2}}
         {\cosh\frac{\lambda_1-\lambda_2}{2}+1}
            \in[1,\frac{1+\cosh\frac{|\lambda_3|}{2}}{2}].
\end{eqnarray*}
\vskip 20pt

\textbf {Case} C. $\Delta_1,\Delta_2$ are geodesic boundaries, $\Delta_3$ is a cone point or a cusp and $N\in\Delta_1$ (see Fig.~\ref{fig:Vcompare2}).
\vskip 5pt
 Using equation (\ref{eq:Pen-1}) for the right angled pentagon pictured in Fig.~\ref{fig:V-pen-2}, we have
$$ \cosh \frac{h_{\Lambda}}{2}=\sinh \frac{\lambda_2}{2} \sinh a ,$$
where $a$ is the length of the geodesic perpendicular to $\Delta_1$ and $\Delta_2$.
\begin{eqnarray*}
  \cosh h_{\Lambda}-1
  &=&2\cosh^2{\frac{h_{\Lambda}}{2}}-2\\
  &=&2\frac{\cosh^2 \frac{\lambda_2}{2}+
        2\cosh \frac{\lambda_1}{2}\cosh \frac{\lambda_2}{2}\cos \frac{|\lambda_3|}{2}+\cos^2\frac{|\lambda_3|}{2}}{\sinh^2 \frac{\lambda_1}{2}},
\end{eqnarray*}
and
\begin{eqnarray*}
  \frac{\cosh h_{\Lambda}-1}{\cosh h_0-1}
  &=& \frac{\cosh^2 \frac{\lambda_2}{2}+
        2\cosh \frac{\lambda_1}{2}\cosh \frac{\lambda_2}{2}\cos \frac{|\lambda_3|}{2}+\cos^2\frac{|\lambda_3|}{2}}{\cosh^2 \frac{\lambda_2}{2}+
        2\cosh \frac{\lambda_1}{2}\cosh \frac{\lambda_2}{2}+1}\\
  &\in& [\cos^2\frac{|\lambda_3|}{2},1].
\end{eqnarray*}
\vskip 20pt

\textbf {Case} D. $\Delta_1,\Delta_2$ are geodesic boundaries,  $\Delta_3$ is a geodesic boundary  and $N\in\Delta_1$.
\vskip 5pt
Replace  $\cos \frac{|\lambda_3|}{2}$ in {case} C  by $\cosh \frac{|\lambda_3|}{2}$, we have
\begin{eqnarray*}
  \frac{\cosh h_{\Lambda}-1}{\cosh h_0-1}
  &=& \frac{\cosh^2 \frac{\lambda_2}{2}+
        2\cosh \frac{\lambda_1}{2}\cosh \frac{\lambda_2}{2}\cosh \frac{|\lambda_3|}{2}+\cosh^2\frac{|\lambda_3|}{2}}{\cosh^2 \frac{\lambda_2}{2}+
        2\cosh \frac{\lambda_1}{2}\cosh \frac{\lambda_2}{2}+1}\\
  &\in& [1,\cosh^2\frac{|\lambda_3|}{2}].
\end{eqnarray*}

\vskip 20pt

\textbf {Case} E. $\Delta_1$ is a geodesic boundary, $\Delta_2,\ \Delta_3$ are  cone points or cusps, and $N\in\Delta_1$ (see Fig.~\ref{fig:Jcompare}).
\vskip 5pt
We choose a quadrilateral with two adjacent vertices at $\Delta_1$ and $\Delta_2$, and two right angles at the endpoints of the opposite side $\xi$ contained in $\Delta_1$ (see the right figure of Fig.~\ref{fig:J-Quadri}).
Let $\xi_2,\xi_3$ be the subsides of $\xi$ divided by $h_{\zeta}$.
From ~(\ref{eq:tri-rec-0}) in Lemma~\ref{lem:ele-formular}, we have
$$\cos \frac{|\lambda_2|}{2}= \sinh \frac{h_{\Lambda}}{2}
  \sinh |\xi_2|,\
\cos \frac{|\lambda_3|}{2}= \sinh \frac{h_{\Lambda}}{2}
  \sinh |\xi_3|. $$
  Set $l_{\Lambda}\triangleq (\sinh\frac{h_{\Lambda}}{2})^{-1} ,$ then
$$\sinh |\xi_2|=l_{\Lambda}\cos \frac{|\lambda_2|}{2},
\ \sinh |\xi_3|=l_{\Lambda}\cos \frac{|\lambda_3|}{2}.$$
Note that $|\xi_2|+|\xi_3|=\lambda_1/2,$ hence
$$\sinh \frac{\lambda_1}{2}=
l_{\Lambda}\cos \frac{|\lambda_2|}{2}
[1+l^2_{\Lambda}\cos^2 \frac{|\lambda_3|}{2}]^{1/2}
+
l_{\Lambda}\cos \frac{|\lambda_3|}{2}
[1+l^2_{\Lambda}\cos^2 \frac{|\lambda_2|}{2}]^{1/2}. $$

Set $$K\triangleq \max\{\cos \frac{|\lambda_2|}{2},\cos \frac{|\lambda_3|}{2}\},\ k\triangleq\min\{\cos \frac{|\lambda_2|}{2},\cos \frac{|\lambda_3|}{2} \}.$$
Then we get
$$
2kl_{\Lambda}[1+k^2l^2_{\Lambda}]^{1/2}
\leq
\sinh \frac{\lambda_1}{2}
\leq
2Kl_{\Lambda}[1+K^2l^2_{\Lambda}]^{1/2},
$$
which means that
$$
\frac{\sinh \frac{\lambda_1}{4}}{K}
\leq
l_{\Lambda}
\leq
\frac{\sinh \frac{\lambda_1}{4}}{k}.
$$
In particular,
$$l_0=\sinh\frac{\lambda_1}{4}.$$
Therefore
\begin{eqnarray*}
  \frac{\cosh h_{\Lambda}-1}{\cosh h_0-1}
  &  =  &
  \frac{\sinh^2 \frac{h_{\Lambda}}{2}}{\sinh^2 \frac{h_{0}}{2}}
          =\frac{l^2_0}{l^2_{\Lambda}}\\
&\in& [k^2,K^2].
\end{eqnarray*}

\vskip 10pt
\textbf {Case} F. $\Delta_1$ is a geodesic boundary, $\Delta_2,\ \Delta_3$ are geodesic boundaries, and $N\in\Delta_1.$
\vskip 5pt
Replace $\cos \frac{|\lambda_2|}{2}$ and $\cos \frac{|\lambda_3|}{2}$ in case E by $\cosh \frac{|\lambda_2|}{2}$ and $\cosh \frac{|\lambda_3|}{2}$, we get
\begin{eqnarray*}
  \frac{\cosh h_{\Lambda}-1}{\cosh h_0-1}
  &  =  &
  \frac{\sinh^2 \frac{h_{\Lambda}}{2}}{\sinh^2 \frac{h_{0}}{2}}
          =\frac{l^2_0}{l^2_{\Lambda}}\\
&\in& [k^2,K^2],
\end{eqnarray*}
where
$$K\triangleq \max\{\cosh \frac{|\lambda_2|}{2},\cosh \frac{|\lambda_3|}{2}\},\ k\triangleq\min\{\cosh \frac{|\lambda_2|}{2},\cosh \frac{|\lambda_3|}{2} \}.$$

\vskip 10pt

\end{proof}
\remark The estimates in  Lemma \ref{lem:Y-compare} are optimal since all the equalities in the inequations can hold simultaneously.

\vskip 10pt
Next, we  compare the lengths of non-peripheral simple closed curves between hyperbolic cone surfaces based on Lemma~\ref{lem:Y-compare}.

\begin{figure}
  \subfigure[]
  {
  \includegraphics[width=35mm]{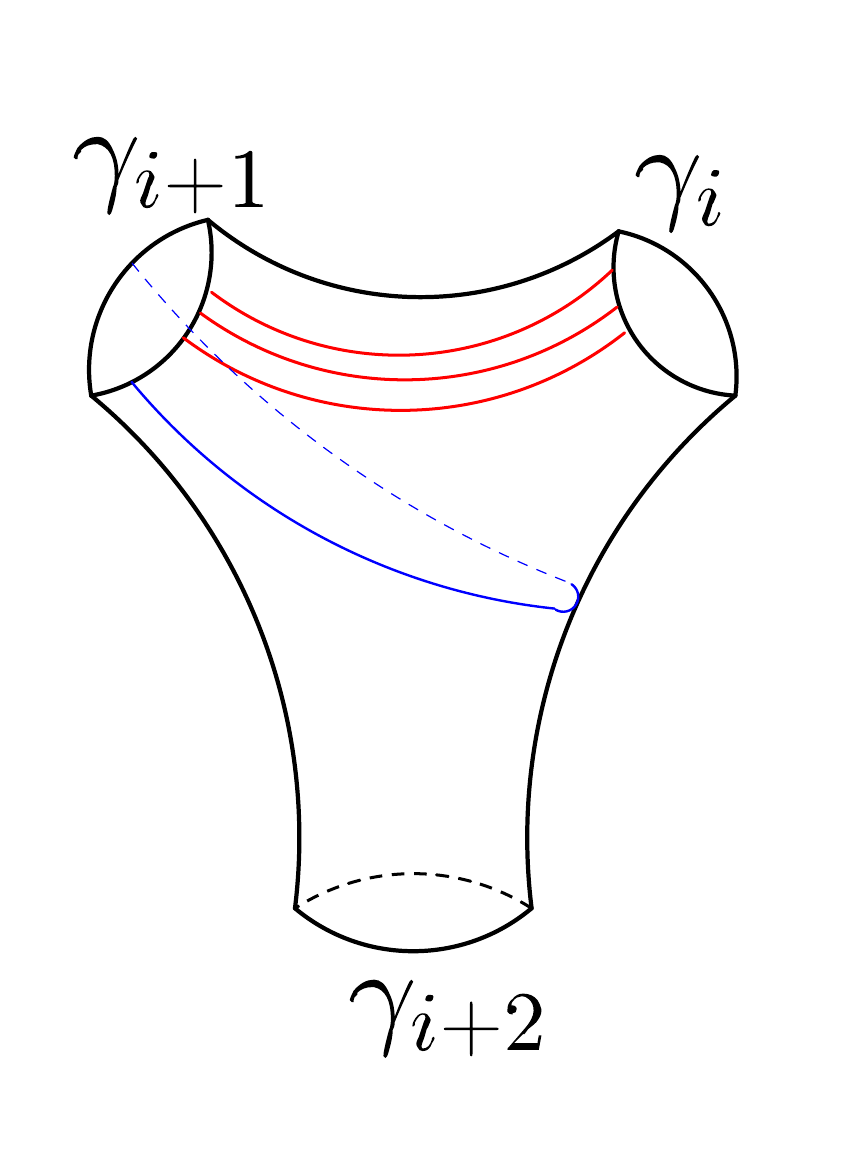}
  \label{fig:Y-arc}
  }
    \subfigure[]
  {
  \includegraphics[width=35mm]{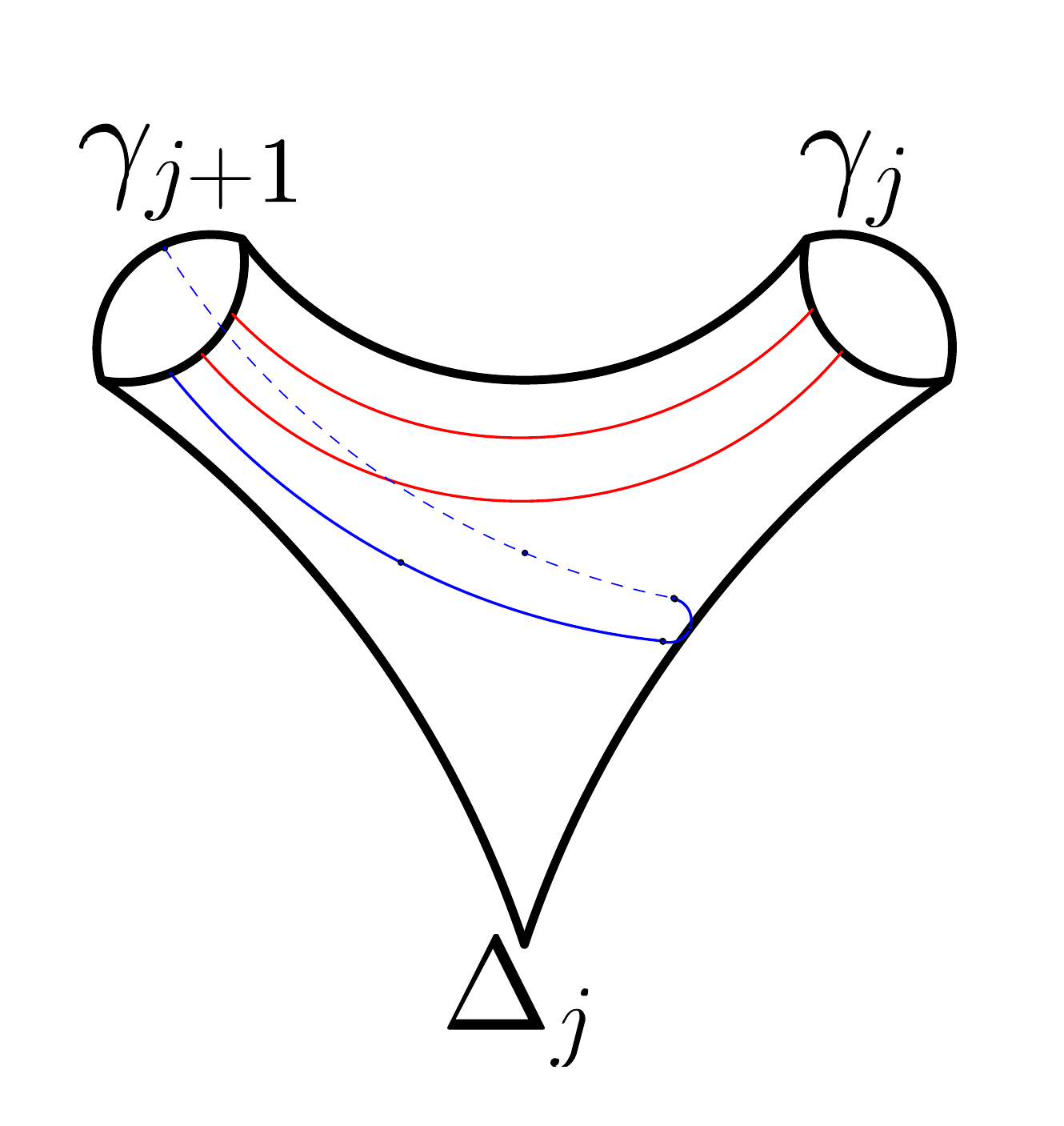}
  \label{fig:V-arc}
  }
    \subfigure[]
  {
  \includegraphics[width=35mm]{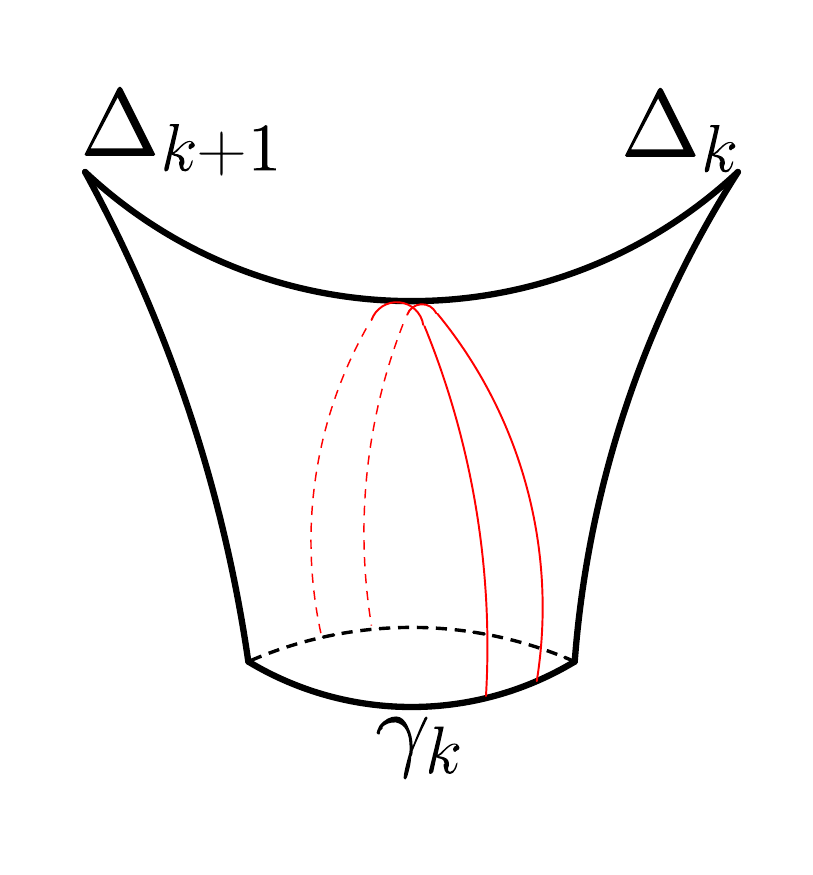}
  \label{fig:J-arc}
  }

  \caption{}
  \label{fig:arc}
\end{figure}
\begin{theoremb}[Length comparison inequalities]\label{thm:com-surface}
Let  $S$ be  a non-exceptional surface. There exist constants $C,D$  depending on $\Lambda$ such that for any $X\in \mathcal{T}_{g,n}(\Lambda)$, $X'=F_{\Gamma,\mathcal B,\Lambda}(X)$ and any isotopy class of non-peripheral simple closed curve $[\alpha]$,
  \begin{equation*}
   \begin{cases}
   |{l_X([\alpha])}-{l_{X'}([\alpha])}|\leq D\sum_{j=1}^{3g-3+n} i([\alpha],[\gamma_j]), \\
  \frac{1}{C}\leq \frac{l_X([\alpha])}{l_{X'}([\alpha])}\leq C,
  \end{cases}
  \end{equation*}
where  $i(\cdot,\cdot)$ is the geometric intersection number, and ${l_X([\alpha])}$ is the length of the geodesic representative in $[\alpha]$.  Moreover, $D\to 0,C\to 1$ as $\Lambda\to 0$.
 \end{theoremb}

 \begin{proof}
     The pants decomposition $\Gamma$ divides the surface $S$ into $2g-2+n$ pairs of topological pants $\mathcal{R}=\{R_1,...,R_{2g-2+n}\}$. Let $R^X_i$ be the restriction of $X$ on $R_i$,  $\alpha^X$ the geodesic representative of $[\alpha]$ on $X$, and $\alpha^X_{i}$ the union of restrictions of $\alpha^X$ on $R^X_i$, $i=1,...,2g-2+n$ (see Fig.~\ref{fig:arc}). The sum of the numbers of arcs in $\alpha_i^{X}$ on all pairs of pants is $\Sigma_{j=1}^{3g-3+n} i([\alpha],[\gamma_j])$. For any different $X,X'\in T_{g,n}(\Lambda)$, the geodesic representative $\alpha^X$ of $\alpha$ on $X$  is usually not a geodesic representative of $\alpha$ on $X'$. We modify $\alpha^X$ as follows: for any $i=1,...,2g-2+n$,  replace each arc in $\alpha_i^X$ by its geodesic representative with respect to $X'$ in its homotopy class relative to the endpoints. Denote by $\alpha^{XX'}$ the resulting simple closed curve   and by  $\alpha_i^{XX'}$ the restriction of $\alpha^{XX'}$ on $R_i^{X'}$.
     Recall that $l_X([\alpha])$ represents the length of the geodesic representative in $[\alpha]$, we denote by $l_X(k)$ the length of any arc $k$ on $X$.
     \vskip 10pt
     (a) From Lemma  ~\ref{lem:Y-compare}, we have
     \begin{eqnarray*}
     l_{X'}([\alpha])-l_X([\alpha])
     &\leq&
     l_{X'}(\alpha^{XX'})-l_X([\alpha])\\
     &=& \Sigma_{i=1}^{2g-2+n}[l_{X'}(\alpha^{XX'}_i)-l_X(\alpha^X_i)]\\
     &\leq& D\Sigma_{j=1}^{3g-3+n} i([\alpha],[\gamma_j]),
     \end{eqnarray*}
     where $D=\max\{D_1,D_2,D_3\}$ and $D_1,D_2,D_3$ are the constants in Lemma ~\ref{lem:Y-compare}.

     By interchanging the role of $X$ and $X'$, we get
     $$ l_{X}([\alpha])-l_{X'}([\alpha])
     \leq D\Sigma_{j=1}^{3g-3+n} i([\alpha],[\gamma_j]).$$

      Moreover, we have  $D\to0$, as $\Lambda\to0$.

     \vskip 10pt

     (b) Again from Lemma  ~\ref{lem:Y-compare}, we have
      \begin{eqnarray*}
      l_{X'}([\alpha])&\leq& l_{X'}(\alpha^{XX'})\\
      &=& \Sigma_{i=1}^{2g-2_n}l_{R^{X'}_i}(\alpha^{XX'}_i)\\
      &\leq& \Sigma_{i=1}^{2g-2_n}Cl_{R^X_i}(\alpha^{X}_i)\\
      &=& C l_X(\alpha^X)\\
      &=&Cl_X([\alpha]),
      \end{eqnarray*}

     and
      \begin{eqnarray*}
      l_{X'}([\alpha])&=& l_{X'}(\alpha^{X'})\\
      &=& \Sigma_{i=1}^{2g-2_n}l_{R^{X'}_i}(\alpha^{X'}_i)\\
      &\geq& \Sigma_{i=1}^{2g-2_n}C^{-1}l_{R^{X}_i}(\alpha^{X'X}_i)\\
      &=& C^{-1} l_{X}(\alpha^{X'X})\\
      &\geq&C^{-1}l_{X}([\alpha]),
      \end{eqnarray*}
      where $C=\max\{C_1,C_2,C_3\}$ and $C_1,C_2,C_3$ are the constants in Lemma ~\ref{lem:Y-compare}.

      Moreover $C\to1$, as $\Lambda\to0$.

 \end{proof}

 \section{Almost-isometry between $\mathcal{T}_{g,n}(\Lambda)$ and $\mathcal{T}_{g,n}(0)$}\label{sec:al-iso}

In this section, we prove Theorem D.
\begin{theoremd}[Almost-isometry]\label{thm:quasi-iso}
  Suppose  $\Lambda\in(-\pi,\infty)^n$ and that   $S$ is a non-exceptional surface.  Then
  $F_{\Gamma,\mathcal B,\Lambda}:\mathcal{T}_{g,n}(\Lambda)\to \mathcal{T}_{g,n}(0)$ is an almost-isometry, i.e. there is a constant $C$ depending only on $\Lambda$ such that
  $$ d_{Th}(X'_1,X'_2)-C\leq d_{Th}(X_1,X_2)\leq d_{Th}(X'_1,X'_2)+C, \text{ for any } X_1,X_2\in\mathcal{T}_{g,n},$$
  where $X'_1=F_{\Gamma,\Lambda}(X_1)$ and $X'_2=F_{\Gamma,\Lambda}(X_2)$.
  Moreover, $C\to0$, as $\Lambda\to0$.
\end{theoremd}
\begin{proof}
  It follows from Theorem B that
  \begin{eqnarray*}
      d_{Th}(X_1,X_2)&=&\sup_{[\alpha]\in \Sim}\log\frac{l_{X_2}([\alpha])}{l_{X_1}([\alpha])}\\
      &\leq& \sup_{[\alpha]\in \Sim}\log\frac{Cl_{X'_2}([\alpha])}{C^{-1}l_{X'_1}([\alpha])}\\
      &=& d_{Th}(X'_1,X'_2)+2\log C,
  \end{eqnarray*}
and
  \begin{eqnarray*}
      d_{Th}(X_1,X_2)&=&\sup_{[\alpha]\in \Sim}\log\frac{l_{X_2}([\alpha])}{l_{X_1}([\alpha])}\\
      &\geq& \sup_{[\alpha]\in \Sim}\log\frac{C^{-1}l_{X'_2}([\alpha])}{Cl_{X'_1}([\alpha])}\\
      &=& d_{Th}(X'_1,X'_2)-2\log C.
  \end{eqnarray*}

\end{proof}
 It follows from Remark \ref{remark:twist} that
 the map $F_{\Gamma,\mathcal B, \Lambda}$ is independent of $\mathcal B$. Instead, it depends on the choice of $\Gamma$, but the images of $X\in\mathcal{T}_{g,n}(\Lambda)$ under  different pants-compositions $\Gamma,\Gamma'$ stay within a bounded distance, i.e. $$d_{Th}(F_{\Gamma,\mathcal B,\Lambda}(X),F_{\Gamma',\mathcal B',\Lambda}(X))\leq 2C,$$
 where $\mathcal B'$ is a collection of seams with respect to $\Gamma'$.

 Indeed, it follows from Theorem B that for any non-peripheral simple closed curve $\alpha$, we have
 $$\frac{1}{C}\leq \frac{l_X([\alpha])}{l_{F_{\Gamma,\mathcal B,\Lambda}(X)}([\alpha])}\leq C$$
 and
 $$\frac{1}{C}\leq \frac{l_X([\alpha])}{l_{F_{\Gamma',\mathcal B',\Lambda}(X)}([\alpha])}\leq C,$$
 where $C$ is the constant from Theorem B.
 Hence,
  $$\frac{1}{C}\leq \frac{l_{F_{\Gamma',\mathcal B',\Lambda}(X)}([\alpha])}{l_{F_{\Gamma,\mathcal B,\Lambda}(X)}([\alpha])}\leq C,$$
  which induces
  $$d_{Th}(F_{\Gamma,\mathcal B,\Lambda}(X),F_{\Gamma',\mathcal B',\Lambda}(X))\leq 2C.$$
 Besides, the composition map $F_{\Gamma,\mathcal B,\Lambda}\circ F^{-1}_{\Gamma',\mathcal B',\Lambda}:\mathcal{T}_{g,n}(0)\to\mathcal{T}_{g,n}(0)$ defines an almost-isometry from $\mathcal{T}_{g,n}(0)$ to itself.


\section{Basic properties of  $\mathcal{T}_{g,n}(\Lambda)$}\label{sec:thurstonb}
In this section, we suppose that $\Lambda\in(-\pi,\infty)^n$ and that the surfaces  are admissible.

\subsection{Topology on $\mathcal{T}_{g,n}(\Lambda)$}
From the discussions in sections \S\ref{sec:met}-\S\ref{sec:rigidity}, we know that the Teichm\"uller space $\mathcal{T}_{g,n}(\Lambda)$ can be equipped with the following topologies:
\begin{itemize}
  \item the topology induced from the functional space $\mathbb R^\mathcal S_+$ over the space of isotopy classes of non-peripheral simple closed curves, where $\mathbb R^\mathcal S_+$ is endowed with the weak topology;
  \item the topology induced from the Fenchel-Nielsen coordinates;
  \item the topology induced from the Thurston metric.
\end{itemize}
It follows from Theorem \ref{thm:mlss}  that the first topology and the second topology are equivalent.  The equivalence between the first topology and the third topology can be obtained from Theorem E.
\subsection{Metric properties of $(\mathcal{T}_{g,n}(\Lambda),d_{Th})$}\label{sec:metric-pro}
Let $\bar{d}_{Th}(X_1,X_2)\triangleq {d}_{Th}(X_2,X_1)$ and
$d_{sym}(X_1,X_2)\triangleq d_{Th}(X_1,X_2)+d_{Th}(X_2,X_1)$. We have the following basic properties.
\begin{proposition}\label{prop:T-ele-prop}
  Suppose  $\Lambda\in(-\pi,\infty)^n$ and that   $S$ is a non-exceptional surface. Let  $(\mathcal{T}_{g,n}(\Lambda),d_{Th})$ be endowed with the topology induced by the metric $d_{Th}$. Then
  \begin{enumerate}[(a)]
    \item $d_{Th}$ is a proper metric on $\mathcal{T}_{g,n}(\Lambda)$.
    \item $(\mathcal{T}_{g,n}(\Lambda),d_{Th})$ is a geodesic metric space.
    \item The topologies on $\mathcal{T}_{g,n}(\Lambda)$ defined by $d_{Th}$, $\bar{d}_{Th}$ and  $d_{sym}$ are the same,
        \\i.e. $d_{Th}(X,X_n)\to 0 \iff d_{Th}(X_n,X)\to 0$ as $n\to\infty$.
  \end{enumerate}
\end{proposition}
\begin{proof}
  (a) Since $\mathcal{T}_{g,n}(\Lambda)$ is homeomorphic to  $(-\pi,\infty)^n\times \mathbb{R}^{3g-3+n}_+\times \mathbb{R}^{3g-3+n}$, a subset of $\mathcal{T}_{g,n}(\Lambda)$ is compact if and only if it is closed and bounded. It follows that $d_{Th}$  is a proper metric.

  (b) Suppose $X_1,X_2\in \mathcal{T}_{g,n}(\Lambda)$, it follows from (a) that the ball $\{X\in \mathcal{T}_{g,n}(\Lambda):d(X_1,X)\leq d(X_1,X_2)\}$ is a compact subset of $\mathcal{T}_{g,n}(\Lambda)$. Now the existence of a geodesic connecting $X_1$ and $X_2$ in $\{X\in \mathcal{T}_{g,n}(\Lambda):d(X_1,X)\leq d(X_1,X_2)\}$ is a standard argument by applying the Ascoli-Arzela Theorem (see \S1.6 in ~\cite{Bu} for instance).

  (c) The proof is similar to that in \cite{Liu1} and \cite{PT1}.

\end{proof}

\subsection{The Thurston's boundary of $\mathcal{T}_{g,n}(\Lambda)$}\label{sec:thurs-boundary}


 A \textit{measured geodesic lamination} is a geodesic lamination equipped with a transverse invariant measure (see \cite{Pen} for the precise definition of measured laminations). The simplest example of a measured geodesic lamination is a weighted simple closed curve. A measured geodesic lamination $(L,\lambda)$ can be viewed as a functions on the set of isotopy classes of non-peripheral simple closed curves via:

 \begin{eqnarray*}
   i_{(L,\lambda)}:\Sim &\longrightarrow & [0,\infty)\\
             {[\alpha]} &\longmapsto &              \inf_{\bar{\alpha}\in[\alpha]}\lambda(\bar{\alpha}).
 \end{eqnarray*}
The value $\inf_{\bar{\alpha}\in[\alpha]}\lambda(\bar{\alpha})$ is called the \textit{intersection  number} of $(L,\lambda)$ and $[\alpha]$. From now on, for simplicity, we just write $\lambda$ instead of $(L,\lambda)$ and denote the intersection number between $\lambda$ and $[\alpha]$ by $i(\lambda,[\alpha])$. Two measured geodesic laminations $\lambda'$ and $\lambda$ are called \textit{equivalent} if $i(\lambda,[\alpha])=i(\lambda',[\alpha])$ for all $[\alpha]\in \Sim$.  Denote by $\mathcal{ML}(S)$ the space of equivalence classes of measured geodesic laminations. The definition of intersection number induces an embedding of $\mathcal{ML}(S)$ into $\mathbb{R}^{\mathcal{S}}_+$, the functional space of $\Sim$. Let $\mathcal{PML}(S)\triangleq\mathcal{ML}(S)/{\mathbb{R}_+}$ be the projection of $\mathcal{ML}(S)$.
\begin{theorem}[\cite{FLP} The Thurston's boundary of $\mathcal{T}_{g,n}(0)$]\label{thm:thurston-bound-p}

  Let $\Psi_{0}$ and $ \Pi$ be the maps defined as following:
  \begin{eqnarray*}
    \Psi_{0}:\mathcal{T}_{g,n}(0)& \longrightarrow &\mathbb{R}^{\mathcal{S}}_+\\
                X &\longmapsto & (l_X([\alpha]))_{\alpha\in\Sim},\\
  \end{eqnarray*}
and
  \begin{eqnarray*}
    \Pi: \mathbb{R}^{\mathcal{S}}_+ &\longrightarrow& P\mathbb{R}^{\mathcal{S}}_+\\
        (s_{\alpha})_{\alpha\in\Sim}&\longmapsto &[(s_{\alpha})_{\alpha\in\Sim}].
  \end{eqnarray*}
  Then
  \begin{enumerate}[(a)]
    \item both $\Psi_{0}$ and $\Pi\circ\Psi_0$ are embeddings, where $\mathcal{T}_{g,n}(0)$ is equipped with the topology induced by $d_{Th}$ and $\mathbb{R}^{\mathcal{S}}_+$ is equipped with the weak topology, i.e. the pointwise convergence topology;
    \item
        the boundary of $\mathcal{T}_{g,n}(0)$ in $P\mathbb{R}^{\mathcal{S}}_+$ is the projection of the space of measured lamination, i.e.
        $\partial \overline{\mathcal{T}_{g,n}(0)}^{Th}=\mathcal{PML}(S)$.

  \end{enumerate}
\end{theorem}

\begin{theoreme}[The Thurston's boundary]\label{thm:thurston-bound}
  Suppose  $\Lambda=(\lambda_1,...,\lambda_n)\in (-\pi,\infty)^n$  and that $S$ is a non-exceptional surface. Let $\Psi_{\Lambda}$ and $ \Pi$ be the maps defined below:
  \begin{eqnarray*}
    \Psi_{\Lambda}:\mathcal{T}_{g,n}(\Lambda)& \longrightarrow &\mathbb{R}^{\mathcal{S}}_+\\
                X &\longmapsto & (l_X([\alpha]))_{\alpha\in\Sim},\\
  \end{eqnarray*}
and
  \begin{eqnarray*}
    \Pi: \mathbb{R}^{\mathcal{S}}_+ &\longrightarrow& P\mathbb{R}^{\mathcal{S}}_+\\
        (s_{\alpha})_{\alpha\in\Sim}&\longmapsto &[(s_{\alpha})_{\alpha\in\Sim}].
  \end{eqnarray*}
  Then the followings hold.
  \begin{enumerate}[(a)]
    \item Both $\Psi_{\Lambda}$ and $\Pi\circ\Psi_{\Lambda}$ are embeddings, where $\mathcal{T}_{g,n}(\Lambda)$ is equipped with the topology induced by $d_{Th}$ and $\mathbb{R}^{\mathcal{S}}_+$ is equipped with the weak topology, i.e. the pointwise convergence topology;
    \item Denote by $\partial \overline{\mathcal{T}_{g,n}(\Lambda)}^{Th}$ the closure of $\mathcal{T}_{g,n}(\Lambda)$ in $P\mathbb{R}^{\mathcal{S}}_+$.
        $\mathcal{T}_{g,n}(\Lambda)\ni X_n \longrightarrow \xi\in\partial \overline{\mathcal{T}_{g,n}(\Lambda)}^{Th}\iff \mathcal{T}_{g,n}(0)\ni F_{\Gamma,\Lambda}(X_n)\longrightarrow \xi\in\mathcal{PML}(S)$,
        as a result,
        the boundary of $\mathcal{T}_{g,n}(\Lambda)$ in $P\mathbb{R}^{\mathcal{S}}_+$ is the projection of the space of measured lamination, i.e.
        $\partial \overline{\mathcal{T}_{g,n}(\Lambda)}^{Th}=\mathcal{PML}(S)$.

  \end{enumerate}
\end{theoreme}

\begin{proof}
  (a) Suppose that there is at least one cone point on the surface which we denote by $\Delta_3$. The proof is almost exactly as it is for Theorem~\ref{thm:thurston-bound-p}  with the addition that injectivity of $\Psi_{\Lambda}$ follows  from Theorem~\ref{thm:mlss} and the injectivity of $\Pi\circ\Psi_0$ requires a new arguement.  Our discussion is divided into two cases.

 \vskip 10pt
 \textbf{Case A} . $S$ has at least two generalized boundary components, say $\Delta_2,\Delta_2'$.  Since $S$ is a non-exceptional surface, there is a generalized $\mathcal X$-piece  on $S$ such that $\Delta_2,\Delta_2'$ are its two generalized boundary components, and the remaining two generalized boundary components, denoted by $\Delta_3,\Delta_3'$ are geodesic boundaries. Without loss of generality, we assume that $\Delta_2,\Delta_2'$ are cone points.
  Let $\gamma$ be a waist of this $\mathcal X$-piece and $\delta$ be  a simple closed curve which separates $\Delta_2,\Delta_2'$ from $\Delta_3,\Delta_3'$.
  Let $\beta$ (resp. $\beta') $ be a simple arc  connecting $\Delta_2$ and $\Delta_2'$ which sits on the pants bounded by $\Delta_2,\Delta_2'$ and $\delta$ (connecting $\Delta_3$ and $\Delta_3'$ which sits on the pants bounded by $\Delta_3,\Delta_3'$ and $\delta$). Hence $\{\delta,\beta,\beta'\}$ consists a coordinate system of curves which gives a Fenchel-Nilsen coordinates of $X$. Let $\delta_n$ be a simple closed curve obtained from $\delta$ by  $n$ times Dehn twist along $\gamma$. Denote by $a_2$ (resp. $a'_2$ ) the geodesic arc perpendicular to both $\Delta_2$  and $\gamma$ (resp. $\Delta'_2$ and $\gamma$).

  Combining  (\ref{eq:dist}) and (\ref{eq:dist3}), we have the following formula
\begin{equation}\label{eq:len-gamman}
\begin{array}{rcl}
\cosh{\frac{l_X([\delta_n])}{2}}&=&\sin(\frac{|\lambda_2|}{2})\sin(\frac{|\lambda_2'|}{2})
\{\cosh {|a_2|}\cosh{|a_2'|}\cosh({t}+n)l_X([\gamma])\\
&&+\sinh{|a_2|}\sinh{|a_2'|}\}
-\cos(\frac{|\lambda_2|}{2})\cos(\frac{|\lambda_2|'}{2}),
\end{array}
\end{equation}
  where $t$ represents the twist of $X$ with respect to $\{\delta,\beta,\beta'\}$.

  So we have
  $$
  \lim_{n\to \infty}\frac{\cosh\frac {l_X([\delta_n])}{2} }{\cosh {nl_X([\gamma])}}=\sin{\frac{|\lambda_2|}{2}}\sin{\frac{|\lambda_2'|}{2}}\cosh |a_2|\cosh |a_2'|,
  $$
  where $|\lambda_2|,|\lambda_2'|$ represent the cone angles at $\Delta_2, \Delta_2'$, respectively.

  Hence
  \begin{equation}\label{eq:limit1}
    \lim_{n\to \infty}e^{ \frac{l_X([\delta_n]) }{2}-{ {nl_X([\gamma])}}}=\sin{\frac{|\lambda_2|}{2}}\sin{\frac{|\lambda_2'|}{2}}\cosh |a_2|\cosh |a_2'|.
  \end{equation}
  The left side and the right side of  (\ref{eq:limit1}) can be viewed as two functions over $ \mathcal{T}_{g,n}(\Lambda)$. Denote these two functions by $F(X),G(X)$ respectively. It follows from (\ref{eq:dist2}) that
  $$ \sinh |a_2|=\frac{\cos \frac{|\lambda_2|}{2}+\cos\frac{|\lambda_2'|}{2}\cosh \frac{l_X(\gamma)}{2} }{\sin \frac{|\lambda_2|}{2}\sinh \frac{l_X(\gamma)}{2}},$$
   which means that $|a_2|$ is a decrease  function of $l_X(\gamma)$. Similarly, $|a_2'|$ is a decrease  function of $l_X(\gamma)$. Therefore $G(X)$ is a decrease  function of $l_X(\gamma)$.
   If there were $X,X'\in  \mathcal{T}_{g,n}(\Lambda)$ and $k>1$ such that $l_{X'}([\alpha])=kl_X([\alpha])$ holds for any $ [\alpha]\in \Sim$, then
  $F(X')>F(X)$. On the other hand, $G(X)$ is a decrease  function of $l_X(\gamma)$, then $G(X')<G(X)$ (note that $|\lambda_2'|$ and $|\lambda_2|$ are fixed). But this is impossible since $F(X)=G(X)$ for any $X\in \mathcal{T}_{g,n}(\Lambda)$.
  \vskip 10pt
  \textbf{Case B}. The surface $S$ has only one generalized boundary component. In this case, if $S$ has genus at least two, we can choose two non-peripheral simple closed curves $\gamma_1,\gamma_2$ which cut $S$ into two pieces, one of which, say $ S_1$,  has one genus and two geodesic boundary components. Let $X'$ be the restriction of $X$ on $S_1$. $X'$  is a hyperbolic surface. Following the method used in the proof of \cite[Prop.7.12]{FLP}, we obtain the injection of $\Pi\circ\Phi$.  It remains to consider the case that $S$ is a torus with only one generalized boundary component $\Delta_3$. Let  $\gamma$, $\beta$ be two simple closed curves on $S$ such that $i([\alpha],[\beta])=1$ (see Fig.~\ref{fig:turos-cone0}), and let $\beta_n$ be the simple closed curve obtained from $\beta$ by $n$ times Dehn twist along $\gamma$. From ~(\ref{eq:torus-length}), we have
  $$
    \lim_{n\to \infty}e^{ {l_X([\beta_n]) }-{ {nl_X([\gamma])}}}=
    \frac{\cosh^2 {\frac{l_X([\gamma])}{2}}+\cos\frac{|\lambda_3|}{2}}{\sinh^2 {\frac{l_X([\gamma])}{2}}},
 $$
  which also implies  that there is no $X'\in \mathcal{T}_{g,n}(\Lambda)$ and $k>0$ such that $l_{X'}([\alpha])=kl_X([\alpha])$ for any $ [\alpha]\in \Sim$ (note that $|\lambda_3|$ is fixed).

  \vskip 10pt

 (b) Suppose that $\mathcal{T}_{g,n}(\Lambda)\ni X_n \longrightarrow \xi\in\partial \overline{\mathcal{T}_{g,n}(\Lambda)}^{Th}$, then for any given $[\alpha]\in\Sim$ there exists $ t_n>0$ such that $t_nl_{X_n}([\alpha])\to \xi([\alpha]) $  as $n\to \infty$.

 We claim that $t_n\to 0$ as $n\to\infty$. Note that $X_n$ leaves every compact subset of $\mathcal{T}_{g,n}(\Lambda)$, it follows from the properness of $\Psi_{\Lambda}$ that there is at least one $[\alpha]\in\Sim$ such that $l_{X_n}([\alpha])\to \infty$ as $n\to \infty$, which proves the claim since $t_nl_{X_n}([\alpha])\to \xi([\alpha]) $ for a finite number $\xi([\alpha]) $.

 Combining  this with the estimates in Theorem B, 
 we have
 \begin{eqnarray*}
  &&\lim_{n\to\infty}t_n l_{F_{\Gamma,\Lambda}(X_n)}([\alpha])\\
  &=& \lim_{n\to\infty}t_n l_{X_n}([\alpha])+\lim_{n\to\infty}t_n [l_{F_{\Gamma,\Lambda}(X_n)}([\alpha])- l_{X_n}([\alpha])]\\
  &=& \lim_{n\to\infty}t_n l_{X_n}([\alpha])\\
  &=&\xi([\alpha]),
 \end{eqnarray*}
  which means that $F_{\Gamma,\Lambda}(X_n)\longrightarrow \xi$ as $n\to\infty$. From Theorem~\ref{thm:thurston-bound-p}, we get    $\partial \overline{\mathcal{T}_{g,n}(0)}^{Th}=\mathcal{PML}(S)$. Therefore $F_{\Gamma,\Lambda}(X_n)\longrightarrow \xi\in\mathcal{PML}(S)$. By interchanging the roles of $X_n$ and $F_{\Gamma,\Lambda}(X_n)$, we get the inverse direction statement, that is, if $F_{\Gamma,\Lambda}X_n\longrightarrow \xi\in\mathcal{PML}(S)$, then
  $X_n\longrightarrow \xi\in\partial \overline{\mathcal{T}_{g,n}(\Lambda)}^{Th}$.

\end{proof}

\end{document}